\documentclass[letterpaper, 10pt, conference]{ieeeconf}
\IEEEoverridecommandlockouts
\usepackage{cite}

   \usepackage[pdftex]{graphicx}
   \graphicspath{{C:/Users/User/Documents/Github/Separation_Principle/ACC/paperACC/jpeg}}
  \DeclareGraphicsExtensions{.jpeg,}
\let\labelindent\relax
\usepackage{amsmath,amssymb}

\usepackage{bm} 
\usepackage{xcolor}
\usepackage{cancel}
\usepackage{booktabs}
\usepackage[shortlabels]{enumitem}

\usepackage{algorithmic}
 \usepackage[caption=false,font=footnotesize]{subfig}

\usepackage{url}

\newcommand{\begit}{\begin{itemize}}
	\newcommand{\eit}{\end{itemize}}
\newcommand{\bseq}{\begin{subequations}}
	\newcommand{\eseq}{\end{subequations}}
\newcommand{\bpat}{\begin{pmatrix}}
	\newcommand{\epat}{\end{pmatrix}}
\newcommand{\bmat}{\begin{bmatrix}}
	\newcommand{\emat}{\end{bmatrix}}
\newcommand{\beq}{\begin{equation}}
\newcommand{\eeq}{\end{equation}}
\newcommand{\bc}{\begin{cases}}
\newcommand{\ec}{\end{cases}}
\newcommand{\beqs}{\begin{equation*}}
\newcommand{\eeqs}{\end{equation*}}

\newcommand{\tm}{\tilde{m}}
\newcommand{\tpi}{\tilde{\pi}}

\newcommand{\rtr}{\rm{tr}}
\newcommand{\FF}{\mathbb{F}}

\newcommand{\LL}{\mathbb{L}}
\newcommand{\MM}{\mathbb{M}}	
\newcommand{\BB}{\mathbb{B}}
\newcommand{\PP}{\mathbb{P}}
\newcommand{\EE}{\mathbb{E}}
\newcommand{\RR}{\mathbb{R}}
\newcommand{\NN}{\mathbb{N}}
\newcommand{\II}{\mathbb{I}}
\newcommand{\HH}{\mathbb{H}}
\newcommand{\hY}{\widehat{Y}}
\newcommand{\hX}{\widehat{X}}
\newcommand{\hM}{\widehat{M}}

\newcommand{\tx}{\tilde{x}}

\newcommand{\hx}{\hat{x}}
\newcommand{\te}{\tilde{e}}

\newcommand{\tY}{\tilde{Y}}

\newcommand{\ind}{\mathbf{1}}

\makeatletter
\newcommand*{\bigconcatenate}{\DOTSB\bigconcatenate@\slimits@}
\newcommand{\bigconcatenate@}{\mathop{\mathpalette\bigconcatenate@@\relax}}
\newcommand{\bigconcatenate@@}[2]{%
  \vcenter{\hbox{%
    \sbox\z@{$\m@th#1\bigoplus$}%
    \resizebox{\wd\z@}{!}{\rotatebox[origin=c]{90}{$\m@th#1\bm{\ominus}$}%
  }}}%
}
\makeatother
\newtheorem{mydef}{\it{Definition}}
\newtheorem{lem}{\it{Lemma}}

\newtheorem{thm}{\it{Theorem}}

\newtheorem{prop}{\it{Proposition}}

\newtheorem{rem}{\it{Remark}}
\newtheorem{pb}{\it{Problem}}

\usepackage{hyperref}

\begin{document}
%
\title{Optimal output-feedback control and separation principle for Markov jump linear systems modeling wireless networked control scenarios}



%
\author{Anastasia Impicciatore$^1$,
Yuriy Zacchia Lun$^2$,
Pierdomenico Pepe$^1$, 
and
Alessandro D'Innocenzo$^1$
\thanks{$^1$Center of Excellence for Research DEWS, Department of Information Engineering, Computer Science and Mathematics, University of L'Aquila, L'Aquila 67100,  Italy.\;{\tt\small anastasia.impicciatore@graduate.univaq.it},     {\tt\small alessandro.dinnocenzo@univaq.it},
       {\tt\small pierdomenico.pepe@univaq.it}}
\thanks{$^2$IMT School for advanced studies Lucca, Lucca, Italy.\;{\tt\small yuriy.zacchialun@imtlucca.it}       }}


\maketitle

\begin{abstract}
The communication channels used to convey information between the components of  wireless networked control systems (WNCSs) are subject to packet losses due to time-varying fading and interference. 
We consider a wireless networked control scenario, where the packet loss occurs in both the sensor–controller link (sensing link), and the controller–actuator link (actuation link). Moreover, we consider one time-step delay mode observations of the actuation link.
While the problems of state feedback optimal control and   stabilizability conditions for systems with one time-step delay mode observations of the actuation link have been  already  solved, we study the optimal output feedback control problem, and we derive a separation principle for the aforementioned wireless networked control
scenario. Particularly, we show that the optimal control problem (with one time step delay in the mode observation of actuation link state) and the optimal filtering problem can be solved independently under a TCP-like communication scheme.  
\end{abstract}


%
\IEEEpeerreviewmaketitle
\section{Introduction}\label{sec:intro}
From the automatic control perspective, the wireless communication
channels are the means to convey information
between sensors, actuators, and computational units of
wireless networked control systems. These communication
channels are frequently subject to time-varying fading and
interference, which may lead to packet losses. In the wireless networked control system (WNCS) literature
the packet dropouts have been modeled either as stochastic
or deterministic phenomena \cite{Hespanha2008}. The proposed deterministic
models specify packet losses in terms of time averages or
in terms of worst case bounds on the number of consecutive
dropouts (see e.g., \cite{heemels2010networked}, \cite{Ding2011}). For what concerns stochastic models,
a vast amount of research assumes memoryless packet drops,
so that dropouts are realizations of a Bernoulli process
(\cite{SCHENATOETAL}, \cite{gupta2009data}, \cite{PajicTAC2011}). Other works consider more general correlated
(bursty) packet losses and use a transition probability matrix
(TPM) of a finite-state stationary Markov channel (see e.g.,
the finite-state Markov modelling of Rayleigh, Rician and
Nakagami fading channels in \cite{SADEGHIETAL2008} and references therein) to
describe the stochastic process that rules packet dropouts (see
\cite{SCHENATOETAL}, \cite{GONCALVES2010}, \cite{BARAS2008}). In these works networked control systems with
missing packets are modeled as time-homogeneous Markov
jump linear systems (MJLSs, \cite{COSTA2005}).\\
When the packet drops affect the communication between the controller and actuator, the controller may know the outcome of the transmission and the state of the channel only after a time-step delay. While the problems of optimal linear
quadratic regulation and stabilizability with one time step delay have been solved in \cite{BARAS2008} and\cite{Lun2019stabilizability}, in this note we focus on {\it optimal-output feedback control} in observation of the operational mode of the system. The problem of output feedback control for Markov jump linear systems has been investigated in \cite{COSTA2005,Vargas2016}, where both the dynamics of the plant and the one of the observer are driven by the same Markov chain, moreover in \cite{COSTA2005,Vargas2016} the delay in the mode observation of the actuation channel is not considered.\\
This article generalizes the results of \cite{Lun2019stabilizability} to double-sided packet loss (see \cite{Ding2011}) as one of the contributions, since we consider that the packet loss occurs in both the sensor–controller link (sensing link), and the controller–actuator link (actuation link).
Moreover, we design the optimal output-feedback controller, that 
can be obtained solving the optimal control problem, and the  optimal filtering problem separately. The main difficulty of our approach can be found in the synthesis of the filtering gain,  accounting for both the operational mode of the sensing channel and also the occurrence of packet losses. In \cite{Mo2013} the filtering problem is solved using the Kalman filter for a single channel modelled by a two state Markov chain (hereafter  MC), that corresponds to simplified Gilbert channel. This result cannot be applied to the general Markov channel that requires {\small $2N$} states with {\small $N\!>\!2$}. Other estimation approaches are {\small $\mathcal{H}_2$} and {\small $\mathcal{H}_{\infty}$} estimation, see \cite{GONCALVES2010}, in which sub-optimal filters are obtained considering cluster
availability of the operational modes. We choose a Luenberger like observer instead of Kalman filter (see \cite{Mo2013}), because the filter dynamics depends just on the current mode of the sensing channel (rather than on the entire past history of modes).
We prove as our main contribution that the separation principle holds also in this case, coherently with the result presented in \cite{COSTA2005}.\\
\noindent The paper is organized as follows. In Section \ref{sec:NCSmodel} we present the networked control system and the information flow of actuation and sensing between the plant and the controller. In Section \ref{sec:LQR} we recall the solution of the optimal control problem in our setting derived in \cite{Lun2019stabilizability}; while in Section \ref{sec:Luobs} we present a Luenberger like observer and we provide a LMI approach to find the solution of the optimal filtering problem. In Section \ref{sec:sep_princ}, we state the separation principle.  
We provide a numerical example in Section \ref{sec:ex} and some concluding remarks in Section \ref{sec:Conclusions}. Proofs of Lemmas and Theorems are reported in the appendix.
\subsection{Notation and preliminaries}
\noindent In the following, 
{\small $\mathbb{N}_{0}$} denotes the set of non-negative integers, while 
{\small $\mathbb{F}$} indicates the set of either real or complex numbers.
The absolute value of a number is denoted by {\small $|\cdot|$}.
We recall that every finite-dimensional normed space over {\small $\mathbb{F}$} 
is 
a Banach space \cite{megginson1998introduction}, 
and denote the Banach space of all bounded linear operators of Banach space 
{\small $\mathbb{X}$} into Banach space {\small $\mathbb{Y}$}, by 
{\small $\mathbb{B}\!\left(\mathbb{X},\mathbb{Y}\right)$}. We set 
{\small $\mathbb{B}\!\left(\mathbb{X},\mathbb{X}\right)\!\triangleq\!
\mathbb{B}\!\left(\mathbb{X}\right)$}. 
The identity matrix of size {\small $n$} is indicated by {\small 
$\mathbb{I}_{n}$}. The operation of transposition is denoted by apostrophe, 
the complex conjugation by overbar, while the conjugate transposition is indicated 
by superscript {\small $^*$}. $\FF_*^{n\times n}$ and $\FF_+^{n\times n}$ denote the sets of Hermitian and positive semi-definite matrices, respectively. 
Let us define, for any positive integer {\small $
\mathbf{C}$}, and for any positive integer {\small $n$}, the following sets:
\begin{small}
\begin{align*}
\HH^{\mathbf{C}n,*}&\triangleq\{\mathbf{K}=[\mathbf{K}_m]_{m=1}^{\mathbf{C}};\mathbf{K}_m\in\FF^{n\times n}_*\},\\
\HH^{\mathbf{C}n,+}&\triangleq\{\mathbf{K}\in\HH^{\mathbf{C}n,*};\mathbf{K}_m\in\FF^{n\times n}_+, m=1,\ldots,\mathbf{C}\}.
\end{align*}
\end{small}We denote by {\small $\rho(\cdot)$} the spectral radius of a square matrix 
(or a bounded linear operator), i.e., the largest absolute value of its eigenvalues, 
and by {\small 
$\left\| \cdot \right\|$} either any vector norm or any matrix norm.
Since for finite-dimensional linear spaces all norms are equivalent
\cite[Theorem 4.27]{Kubrusly2001elements} from a topological viewpoint, as vector 
norms we use variants of vector $p$-norms.
For what concerns the matrix norms, we use 
{\small $\ell_1$} and {\small $\ell_2$} norms \cite[p.~341]{horn2012matrix},
that treat {\small $n_r\!\times\! n_c$} matrices as vectors of size
{\small $n_r n_c$}, and use one of the related $p$-norms. The definition of 
{\small $\ell_1$} and {\small $\ell_2$} norms is based on the operation of 
vectorization of a matrix, {\small $\mathrm{vec}(\cdot)$}, which is further used 
in the definition of the operator {\small $\hat{\varphi}(\cdot)$}, to be applied 
to any block matrix, e.g.,~{\small ${\bm \Phi}\!=\!\big[\Phi_{_{\!\!\;m}}\big]_{m=1}^{\mathbf{I}}$}: 
\begin{small}
\begin{equation*}
\hat{\varphi}\!\left({\bf\Phi}\right)\!\!\,\triangleq\!\!\,
\left[\mathrm{vec}\!\left({\Phi}_{_{\!\!\;1}}\!\!\;\right)\!,\dots,\mathrm{vec}\!\left({\Phi}_{_{\!\!\;\mathbf{I}}}\!\!\;\right)\right]'\!.
\end{equation*}
\end{small}The linear operator {\small $\hat{\varphi}(\cdot)$} is a uniform homeomorphisms, 
its inverse operator {\small $\hat{\varphi}^{-1}(\cdot)$} 
is uniformly continuous \cite{naylor2000linear},
and any bounded linear operator 
in {\small $\mathbb{B}\!\left(\FF^{Nn_r\times n_c}\right)$} can be 
represented in {\small $\mathbb{B}\!\left(\FF^{Nn_{r}n_c}\right)$} trough 
{\small $\hat{\varphi}(\cdot)$}.
We denote by {\small $\otimes$} the Kronecker product defined in the 
usual way, see e.g.,  \cite{brewer1978kronecker}, and by $\oplus$ the direct sum. 
Notably, the direct sum of a sequence of square matrices 
{\small $[ \Phi_i ]_{i = 1}^{N}$} produces a block diagonal matrix, 
having its elements,  {\small $\Phi_i$}, on the main diagonal blocks. 
Then, {\small $\mathrm{tr}\left(\cdot\right)$} indicates the trace of a square 
matrix. For two Hermitian matrices of the same dimensions, 
{\small $\Phi_{_{\!1}}$} and {\small $\Phi_{_{\!2}}$}, 
{\small $\Phi_{_{\!1}}\!\succeq\!\Phi_{_{\!2}}$} (respectively 
{\small $\Phi_{_{\!1}}\!\succ\!\Phi_{_{\!2}}$}) means that 
{\small $\Phi_{_{\!1}}\!-\!\Phi_{_{\!2}}$} is positive semi-definite
(respectively positive definite). 
Finally, {\small $\mathbb{E}\!\left(\cdot\right)$} stands for the mathematical 
expectation of the underlying scalar valued random variable.\\
\section{NETWORKED CONTROL SYSTEM MODEL}\label{sec:NCSmodel}
Let us consider a closed loop system where the information flow between controller and plant is sent over a (wireless) communication network under a TCP-like protocol.
The {\it plant} is modeled through a linear stochastic system
\beq\label{mathcalG}
\mathcal{G}:\bc x_{k+1}=Ax_k + \nu_k B u_k + Gw_k,\\
y^s_k=Lx_k + Hw_k,\;y^c_k=\gamma_k y^s_k,\\
z_k=Cx_k +\nu_k Du_k,\\
x_0\in\mathbb{F}^{n_x};
\ec
\eeq
\noindent where $x_k\in \FF^{n_x}$ is the system state, $u_k\in\FF^{n_u}$ is the desired control input, computed by the controller and sent to the actuator. The system is affected by intermittent control inputs and observations due to the occurrence of packet losses on the actuation link, and on the sensing link, respectively. Two binary random variables, {\small $\nu_k$} and {\small $\gamma_k$}, depict these features in the model $\mathcal{G}$. Particularly, the stochastic variable $\nu_k$ models the occurrence of packet losses on the actuation link, while the stochastic variable $\gamma_k$ models the occurrence of packet losses on the sensing link. The vector  $y^s_k\in\FF^{n_y}$ contains the measurements that are sent from the sensor to the controller, while $y^c_k\in\FF^{n_y}$ is the vector received by the controller. If the packet containing $y^s_k$ is correctly delivered, then $y^c_k=y^s_k$; otherwise the controller does not receive the packet and we have $y^c_k=0$. 
The vector $z_k\in\FF^{n_z}$ is the output of the system, that is used to define performance index of the optimal controller.
The sequence $\{w_k\in\FF^{n_w};k\in\mathbb{N}\}$ is a white noise sequence, representing discrepancies between the model and the real process, due to unmodeled dynamics or
disturbances and measurement noise. The noise $w_k$ is assumed to be independent from the initial state $x_0$ and from the stochastic variables  $\nu_k$ and $\gamma_k$, respectively. Specifically, we have {\small $\forall k,l \in \mathbb{N}, k \neq l$} that:
\beq\label{noisemeanvari}
\EE[w_k]=0,\quad \EE[w_k w_k^*]=\mathbb{I}_{n_w},\quad\EE[w_k w_l^*]=0.
\eeq
\indent
As in \cite[Section 5.2]{COSTA2005}, without loss of generality we assume that the system matrices are constant matrices of appropriate sizes, such that 
\beq\label{systemmatrices}
GG^*\!\succeq\!0,\,  GH^*\!=\!0,\,
HH^*\!\succ\!0,\, C^*D\!=\!0,\,D^*D\!\succ\!0.
\eeq

\noindent To describe the stochastic characteristics of variables {\small $\nu_k$} and {\small $\gamma_k$} we use the Markov channel model of the packet dropout process proposed in \cite{Lun2020ontheimpact}, where the transition probabilities between the communication channel's states and the associated probabilities of the packet loss are derived analytically by taking into account the geometry of the propagation environment, the degree of motion around the communicating nodes and the relevant physical phenomena involved. In this model the states of the communication channel are measured through the signal-to-noise-plus-interference ratio (SNIR), and each state is associated with a certain packet error probability (PEP). Formally, consider the stochastic basis
{\small $\left(\Omega, \mathcal{F},\{\mathcal{F}_k	\}_{k\in\NN},\mathbb{P}\right)$}, where {\small$\Omega$} is the sample space, {\small$\mathcal{F}$} is the {\small$\sigma$}-algebra of (Borel) measurable events, {\small$\{\mathcal{F}_k	\}_{k\in\NN}$} is the related filtration and {\small$\mathbb{P}$} is the probability measure. The sensing and control channel states are the output of the discrete-time time-homogeneous Markov chains (MCs): \mbox{\small$\mathbb{\eta}:\NN\times \Omega\to \mathbb{S}_{\eta}\subseteq\NN$} and {\small$\mathbb{\theta}:\NN\times \Omega\to \mathbb{S}_{\theta}\subseteq\NN$}. Indeed, {\small$\{\eta_k\}_{k\in\NN}$}
 and {\small$\{\theta_k\}_{k\in\NN}$} take values in the finite sets {\small$\mathbb{S}_{\eta}=\{1,\ldots, \mathbf{I}\}$} and {\small$\mathbb{S}_{\theta}=\{1,\ldots, \mathbf{N}\}$} and have {\it time-invariant} transition probability matrices (hereafter TPMs) $P=[p_{ij}]_{i,j=1}^\mathbf{N}$ and $Q=[q_{mn}]_{m,n=1}^\mathbf{I}$, respectively. The entries of the TPMs $P$ and $Q$ are defined as:
\begin{subequations}
\begin{align}
p_{ij}&\triangleq\mathbb{P}\left(\theta_{k+1}=j|\theta_k =i\right),\;\forall i,j\in \mathbb{S}_{\theta},\label{TPMentriesa}\\
q_{mn}&\triangleq \mathbb{P}\left(\eta_{k+1}=n|\eta_k =m\right),\;\forall m,n\in \mathbb{S}_{\eta},
\label{TPMentriesb}
\end{align}
satisfying:
\beq\label{TPMentriesc}
\sum_{n\in\mathbb{S}_{\eta}} q_{mn}=1,\; \sum_{j\in \mathbb{S}_{\theta}} p_{ij}=1,\;\forall m\in \mathbb{S}_{\eta},\;\forall i\in \mathbb{S}_{\theta}.
\eeq
\end{subequations}
\noindent The variable {\small$\tpi_i(k)$} denotes the probability that the  MC {\small $\mathbb{\theta}$} is in the mode {\small $i\in\mathbb{S}_{\theta}$} at time {\small $k\in\NN$}, i.e. {\small $\tpi_i(k)=\PP\left(\theta_k=i\right)$}, while {\small $\pi_m(k)$} denotes the probability {\small $\PP\left(\eta_k=m\right)$}.
\noindent In order to provide the stability analysis and control synthesis as follows in this paper, we need to consider
the aggregated state {\small $(\nu_k,\theta_k)$, }i.e. a $2\mathbf{N}$-ary random quantity, where {\small$\nu_k$} accounts for the occurrence of packet losses on the actuation channel. Thus, {\small$\nu_k=0$} if the control packet is lost and {\small $\nu_k=1$} if the control packet is correctly delivered. For this reason, we can write that {\small $\nu_k\in\mathbb{S}_{\nu}\triangleq \{0,1\},\;\forall k\in \NN$}. The Markov chain {\small $\mathbb{\theta}$} describes the evolution of the actuation channel. 
Let us define {\small $\nu_{\theta_k}\!\triangleq\!(\nu_k,\theta_k)$}. The probability of having a successful packet delivery on the control channel depends on the current mode of the channel, that is given by {\small $\theta_k=i$}, i.e.,
\begin{align}
\mathbb{P}(\nu_k =1|\theta_k=i)=\hat{\nu}_i,
\;
\quad \mathbb{P}(\nu_k =0|\theta_k=i)=1-\hat{\nu}_i,
\end{align}
\noindent are the probability that the packet is correctly delivered at time {\small $k\in \NN$}, and the probability of having a packet loss  conditioned to {\small $\theta_k=i$}, respectively.
As far as the sensing channel is concerned, its evolution is described by the MC $\eta_k$. The variable $\gamma_k$ accounts for the occurrence of packet losses on the sensing channel. Thus, {\small$\gamma_k=0$} if the 
sensing packet is lost and {\small $\gamma_k=1$} if the  packet is correctly delivered, i.e $\gamma_k\in\mathbb{S}_{\gamma}\triangleq\{0,1\},\;\forall k\in\NN$. 
\noindent  Thus, we need again an aggregated state {\small$\gamma_{\eta_k}\!\triangleq\!(\gamma_k,\eta_k),\;k\in\NN$}, which is a $2\mathbf{I}$-ary random quantity. Similarly, the probability of having a successful packet delivery on the sensing channel also depends on the current mode of the channel, {\small $\eta_k=n$}, i.e.,
\begin{align}
\mathbb{P}(\gamma_k=1|\eta_k=n)=\hat{\gamma}_n,
\;
\quad \mathbb{P}(\gamma_k=0|\eta_k=n)=1-\hat{\gamma}_n,
\end{align}
\noindent are the probability that the packet is correctly delivered at time {\small $k\in \NN$}, and the probability of having a packet loss  conditioned to {\small $\eta_k=n$}, respectively.
\noindent We can write the system presented in \eqref{mathcalG} as:
\begin{small}
\beq
\bc\label{mathcalG1}
x_{k+1}=Ax_k+\nu_{\theta_k}Bu_k+ Gw_k,\\
y_k=\gamma_{\eta_k}Lx_k+\gamma_{\eta_k}Hw_k,\\
z_k=Cx_k+\nu_{\theta_k}Du_k.
\ec
\eeq
\end{small}Similarly to \cite[Section 5.3]{COSTA2005}, we make the following technical assumptions:
\noindent
\begin{enumerate}[(i)]
\item the initial conditions {\small $x_0, \theta_0, \eta_0$} are independent random
variables,
\item the white  noise {\small $w_k$} is independent from the initial conditions {\small $(x_0,\nu_{\theta_0},\gamma_{\eta_0})$} and from the Markov processes {\small $\nu_{\theta_k},	\;\gamma_{\eta_k}$}, for all values of the discrete time {\small $k$},
\item the sequence {\small $\{w_k;\; k\in \NN \}$} and the
Markov chains {\small$\{\theta_k;\;k\in\NN \}$}, {\small$\{\eta_k;\;k\in\NN \}$} are independent sequences,
\item the MCs {\small $\{\theta_k\}_{k\in\NN}$ and $\{\eta_k\}_{k\in\NN}$} are {\it ergodic}, with {\it steady state probability distributions}
\begin{small} 
\beq
\tpi^{\infty}_i=\lim_{k\to\infty}\tpi_i(k),
\quad
\pi^{\infty}_m=\lim_{k\to\infty}\pi_m(k),
\eeq
\end{small}
respectively. Consequently, the Markov processes $\nu_{\theta_k}$
and $\gamma_{\eta_k}$ are also {\it ergodic}.
\end{enumerate}
\noindent In \cite{Lun2019stabilizability}, the control input {\small $u_{k}$} is designed exploiting the available information regarding the actuation channel state, that is {\small $\theta_{k-1}$}, affected by one time-step delay. For this reason, we will consider, as in \cite{Lun2019stabilizability}, the aggregated Markov state {\small $(\theta_k,\theta_{k-1})$}. The introduced memory is, however, fictitious, since the aggregated MC obeys to the Markov property of the memoryless chain {\small $\mathbb{\theta}$}. Moreover, we are able to compute the probabilities related to the joint process {\small $\left(\nu_k,\theta_k,\theta_{k-1}\right)$}, as in \cite{Lun2019stabilizability}.
As far as the joint process {\small $\gamma_{\eta_k}=\left(\gamma_k,\eta_k\right)$} is concerned,  applying {\it Bayes Law, the Makov poperty, and the independence between {\small $\gamma_{k}$} and {\small $\eta_k$}}, we obtain:
\begin{subequations}
\beq
\mathbb{P}\left(\gamma_{k+1}=1,\eta_{k+1}=n|\eta_k=m \right)=\hat{\gamma}_n q_{mn},
\eeq
\begin{align}
& \mathbb{P}\left(\gamma_{k+1}=0,\eta_{k+1}=n|\eta_k=m \right)=\left(1-\hat{\gamma}_n\right) q_{mn},\\
& \forall m,n \in \mathbb{S}_{\eta}.\nonumber
\end{align}
\end{subequations}
\noindent In order to apply the usual definition of the mean square stability
\cite[pp. 36--37]{COSTA2005} to system \eqref{mathcalG1}, we consider the
operational modes of system \eqref{mathcalG1}, given by 
{\small $\left(\nu_{\theta_k},\theta_{k-1},\gamma_{\eta_k} \right)$}, which is a
{\small $4\mathbf{N}^2\mathbf{I}$}-ary random quantity.
\begin{mydef}\label{def:mss}
A MJLS \eqref{mathcalG1} is \textit{mean square stable} if  there exist equilibrium points
{\small $\mu_e$} and {\small $Q_e$} (independent from initial conditions)
 such that, for any 
initial condition {\small $\left(x_0,\nu_{\theta_0},\gamma_{\eta_0}\right)$}, the following equalities hold: 
\vspace*{-2mm}
\begin{small}
\begin{equation}\label{eq:mss_general}
\lim_{k \to \infty}\left\|\mathbb{E}\!\left(
	x_k \right)\!-\!\mu_e\right\|\!=\!0, \quad
\lim_{k \to \infty}\left\|\mathbb{E}\!\left(
	x_k x_k^* \right)\!-\!Q_{e} \right\| \!=\! 0.
\end{equation}
\end{small}
\end{mydef}
\begin{figure}[thpb]
      \centering
      \includegraphics[scale=0.35]{Figure1.png}
      \caption{Information flow timing between the plant and the controller.}
      \label{fig:1}
   \end{figure}
Fig. \ref{fig:1} illustrates the information flow of actuation and sensing data between the plant and the controller under TCP-like protocols, with a sampling period {\small $T$}. At time {\small $(k-1)T+\delta_3+\Delta+\delta_1$}, the information set \eqref{infoset1} is available to the controller, for the computation of the control input {\small $u_k$} that will be applied at time indexed by {\small $k$}:
\beq\label{infoset1}
\mathcal{F}^{k}=\{(u_t)_{t=0}^{k-1}, (y_t)_{t=0}^{k-1}, (\nu_{\theta_t})_{t=0}^{k-1}, (\gamma_{\eta_t})_{t=0}^{k-1} \}.
\eeq 
\noindent We aim to design a {\it dynamical linear output feedback controller} having the following Markov jump structure: 
\beq\label{K1}
\mathcal{G}_K:\bc
\hx_{k+1}=\widehat{A}(\nu_{\theta_k},\theta_{k-1},\gamma_{\eta_k}) \hx_{k} +\widehat{B}_{\eta_{k}}y_{k},\\
u_{k}= \widehat{C}_{\theta_{k-1}}\hx_{k}. \ec
\eeq
The control problem consists in finding the optimal matrices {\small $\widehat{A}(\nu_{\theta_k},\theta_{k-1},\gamma_{\eta_k}),\widehat{B}_{\eta_{k}},\widehat{C}_{\theta_{k-1}},$} such that the closed-loop system is {\it mean square stable}, according to Definition \ref{def:mss}.
The matrices {\small $\widehat{C}\!=\![\widehat{C}_{l}]_{l=1}^{\mathbf{N}}$}, and {\small $\widehat{B}\!=\![\widehat{B}_{n}]_{n=1}^{\mathbf{I}}$} are the solutions of the optimal control problem and of the optimal filtering problem, respectively.
\section{THE OPTIMAL LINEAR QUADRATIC REGULATOR}\label{sec:LQR}
In this section, we need to exploit the  definition given in \cite[{\it Definition 1}]{Lun2019stabilizability}, dealing with {\it mean square stabilizability} of Markov jump linear systems with one time-step delayed operational mode observations. 
\begin{mydef}[{\it Mean square stabilizability with delay}]\label{Def:MSstab}
The system described by \eqref{mathcalG1} is {\it mean square stabilizable with one time step delay operational mode observation} if for any initial condition $(x_0,\theta_0)$, there exists a {\it mode-dependent gain} $F=[F_l]_{l=1}^\mathbf{N}$, such that  $u_k=F_{\theta_{k\!-\!1}}x_k$ is {\it the mean square stabilizing state feedback} for \eqref{mathcalG1}.
\end{mydef}
\subsection{The Control CARE}\label{subsec:CCARE}
\noindent In this subsection, we compute the {\it optimal mode-dependent control gain with one time-step delayed operational mode observation in the actuation channel}, denoted by {\small $F=[F_l]_{l=1}^\mathbf{N}$}. 
We recall the {\it infinite horizon optimal control problem}, whose solution is given in \cite{Lun2019stabilizability} starting from the more general result presented in \cite{BARAS2008}. We set for any $X\in\HH^{\mathbf{N}n_x,+}$:
\begin{small}
\begin{align*}
\mathcal{A}_l\!\triangleq\!
A^*\!\left(\!\sum_{i=1}^\mathbf{N} p_{li}X_{i}\!\right)\!A\!+\!C^*C,\qquad \mathcal{C}_l\triangleq
A^*\!\left(\!\sum_{i=1}^\mathbf{N} p_{li}\hat{\nu}_i X_{i}\!\right)\!B,\\
\tilde{\mathcal{B}}_l\!\triangleq
\!\sum_{i=1}^\mathbf{N} p_{li}\hat{\nu}_i \!\left(B^* X_{i} B\!+\!D^* D\!\right)\!,\;\; \mathcal{X}_l\!\triangleq\!\mathcal{A}_l\!-\!\mathcal{C}_l\tilde{\mathcal{B}}_l^{-1}\mathcal{C}_l^*,\, l\in\mathbb{S}_{\theta}.
\end{align*}
\end{small}We call the set of equations {\small$X_l\!=\!\mathcal{X}_l(X)$} {\it Control Coupled Algebraic Riccati Equation (hereafter Control CARE)}. Clearly, the necessary condition for the existence of the {\it mean square stabilizing} solution {\small$\tilde{X}\!\in\!\HH^{\mathbf{N}n_x,+}$}, of the Control CARE, is the {\it mean square stabilizability with one time-step delay} of system \eqref{mathcalG1}, according to {\it Definition} \ref{Def:MSstab}. If {\small $\tilde{X}\!\in\!\HH^{\mathbf{N}n_x,+}$} is the {\it mean square stabilizing} solution of the {\it Control} CARE, then the state feedback control input {\small $F_{\theta_{k-1}}x_k$} stabilizes the system in the mean square sense, with one time-step delay in the observation of the actuation channel mode. The solution of the optimal control problem can be computed using the LMI approach presented in \cite{BARAS2008}.
The optimized performance index is given by
\begin{small}
\beq 
J_c=\limsup\limits_{t\to\infty}\frac{1}{t}\EE\left[\sum_{k=0}^t\!\left(z_kz^*_k\right) |\mathcal{F}^k\right],
\eeq 
\end{small}while the performance index achieved by the {\it optimal control law} is 
\begin{small}
\beqs
J_c^*\!=\!{\displaystyle\sum_{i=1}^\mathbf{N}} \tilde{\pi}^{\infty}_i\mathrm{tr}\left(G^*X_{i} G\right).
\eeqs
\end{small}
\section{THE LUENBERGER LIKE OBSERVER}\label{sec:Luobs}
\noindent In this section, we present the filtering problem. To find the output-feedback controller using the information set described in Section \ref{sec:NCSmodel}, we design a {\it Luenberger like observer}, given by:
\begin{small}
\beq\label{tildemathcalG}
\tilde{\mathcal{G}}:\bc\tx_{k+1}\!=\!A\tx_{k}\!+\!\nu_{\theta_{k}}Bu_{k}\!-\!M_{\eta_{k}}(y_{k}\!-\!\gamma_{\eta_k}L\tx_{k}),\\
u_k\!=\!F_{\theta_{k-1}}\tx_k,\\
\tx(0)\!=\! \tx_0;\ec
\eeq
\end{small}where {\small $[ M_m]_{m=1}^{\mathbf{I}}\!\triangleq\!M$} is the {\it mode-dependent filtering gain} to be found as a solution of the filtering problem. It may be noted that when we compute {\small $\tx_{k+1}$} we know whether the control packet and the measurement packet arrived or not at the previous step. Indeed, this information will be contained by {\small $\mathcal{F}^{k+1}$} that will be used to compute the control input to apply at time {\small $k\!+\!1$}, that is {\small $u_{k+1}\!=\!F_{\theta_k}\tx_{k+1}$}.
Let us define the {\it estimate error} as {\small $
\te_{k}\!\triangleq\! x_{k}\!-\!\tx_{k}$}. Consequently, the {\it error dynamics is obtained as follows}:
\begin{align}\label{nexterrordynamics}
\te_{k+1}=\left(A\!+\!\gamma_{\eta_k}M_{\eta_k}L\right)\te_k\!+\!\left(G+\gamma_{\eta_k}M_{\eta_k}H\right)w_k.
\end{align}
\begin{rem}\label{rem:independent gains}
The error dynamics does not depend on the control input. Thus, the gain matrices {\small $F_{\theta_{k-1}}$} and {\small $M_{\eta_k}$} can be computed independently.
\end{rem}
\subsection{Observer stability analysis}
\noindent In this subsection, we provide a stability analysis for the  error system.
We want to find {\it recursive difference equations }for the first moment and the second moment error, {\small $\te_k$}. Specifically, we define:
\begin{subequations}
\beq\label{nextmnkdef}
\tm_n(k)\!\triangleq\!\EE\!\left[\te_k\ind_{\{\eta_{k\!-\!1}\!=n\}}\!\right],\;\tm(k)\!\triangleq\!\bmat \tm_n(k)\!\emat_{n=1}^\mathbf{I}\!\in\!\FF^{\mathbf{I}n_x},
\eeq 
\beq\label{nextYnkdef}
\tY_n(k)\!\triangleq\!\EE\!\left[\te_k\te_k^*\ind_{\{\eta_{k\!-\!1}\!=n\}}\!\right],\; \tY(k)\!\triangleq\!\bmat\tY_n(k)\!\emat_{n=1}^\mathbf{I}\!\in\!\FF^{\mathbf{I}n_x\!\times\!n_x}.
\eeq
\end{subequations}
\noindent So that the {\it first and  second moment} of {\small $\te_k$} are given by:
\beq\label{nextEe}
\EE\left[\te_k\right]=\sum_{n=1}^{\mathbf{I}}\tm_n(k),\quad\EE\left[\te_k\te_k^*\right]=\sum_{n=1}^{\mathbf{I}}\tY_n(k).
\eeq
\noindent In order to carry on the mean square stability analysis, in the spirit of \cite{COSTA2005}, we need to define the operators {\small $
\mathcal{V}(\cdot)\!\triangleq\! [\mathcal{V}_m(\cdot)]_{m=1}^\mathbf{I}$, $
\tilde{\mathcal{V}}(\cdot)\!\triangleq\![ \tilde{\mathcal{V}}_m(\cdot)]_{m=1}^\mathbf{I}$}, and {\small $\mathcal{J}(\cdot)\!\triangleq\![\mathcal{J}_m(\cdot)]_{m=1}^\mathbf{I}$}, all in {\small $\BB\left(\FF^{\mathbf{I}n_x\times n_x}\right)$}, as follows. For all {\small $\mathbf{S}=[ S_m]_{m=1}^{\mathbf{I}}$}, {\small $\mathbf{T}\!=\![ T_m]_{m=1}^{\mathbf{I}}$}, both in {\small $\FF^{\mathbf{I}n_x\!\times \!n_x}$}, we specify the inner product as:
\beq\label{intprod}
\left\langle \mathbf{S};\mathbf{T}\right\rangle\!\triangleq\!\sum_{m=1}^\mathbf{I}\! \rtr \left(S_m^*T_m\right),
\eeq
\noindent while the components of operators $\mathcal{V}$, $\mathcal{J}$, and $\tilde{\mathcal{V}}$ are defined for $n\!\in\!\mathbb{S}_{\eta}$ by: 
\begin{subequations}
\beq\label{mathcalD}
\mathcal{D}_n(\mathbf{S})\triangleq \sum\limits_{m=1}^{\mathbf{I}}q_{mn} S_m,
\eeq
\beq\label{mathcalV}
\mathcal{V}_n(\mathbf{S})\!\triangleq\!\hat{\gamma}_n \Gamma_{n1}\mathcal{D}_n(\mathbf{S})\Gamma_{n1}^*\!+\!(1\!-\!\hat{\gamma}_n)\Gamma_{n0}\mathcal{D}_n(\mathbf{S})\Gamma_{n0}^*,
\eeq
\beq\label{mathcalJ}
\mathcal{J}_m(\mathbf{S})\!\triangleq\! {\displaystyle \sum_{n=1}^{\mathbf{I}}}\!q_{mn}\hat{\gamma}_n\Gamma_{n1}^*S_n\Gamma_{n1}\!\!
+\!\!{\displaystyle\sum_{n=1}^\mathbf{I}}\!q_{mn}(1\!-\!\hat{\gamma}_n)\Gamma_{n0}^*S_n\Gamma_{n0},
\eeq
\beq\label{barmathcalV}\tilde{\mathcal{V}}_n(\mathbf{S})\!=\! \hat{\gamma}_n \Lambda_{n1}\mathcal{D}_n(\mathbf{S})\Lambda_{n1}^*\!+\!(1\!-\!\hat{\gamma}_n)\Lambda_{n0}\mathcal{D}_n(\mathbf{S})\Lambda_{n0}^*,
\eeq
\end{subequations}where the matrices {\small $\Gamma_{n1},\,\Gamma_{n0},\,\Lambda_{n1},\,\Lambda_{n0}\!\in\!\FF^{n_x\times n_x}$}  will be defined later in the paper.
\begin{rem}\label{remrhov}
Clearly, we have that {\small $\left(\mathcal{V}(\mathbf{S})\right)^*\!\!=\!\!\mathcal{V}(\mathbf{S}^*)$}, and it is immediate to verify (starting from \eqref{intprod}, applying \eqref{mathcalV}, \eqref{mathcalJ}, linearity of the trace operator and its invariance under cyclic permutations) that {\small $\mathcal{J}$} is the {\it adjoint operator } of {\small $\mathcal{V}$}, i.e. {\small $\mathcal{V}^*\!=\!\mathcal{J}$}. This is a specialization of \cite[Prop. 3.2, p. 33]{COSTA2005}. Furthermore, it is evident from their definitions \eqref{mathcalV}, and \eqref{mathcalJ}, that {\small $\mathcal{V}$ and $\mathcal{J}$} are Hermitian and positive operators.
\end{rem}
\begin{prop}\label{propmeanerror}
Consider the {\it error system} \eqref{nexterrordynamics}. Then, the following equalities hold:
\beq\label{nexttmkexp}
\tm(k\!+\!1)=\mathcal{B}\tm(k),\ \tY(k\!+\!1)=\mathcal{V}(\tY(k))+\mathcal{O}(M,k),
\eeq
\begin{subequations}
\beq\label{mathcalB}
\mathcal{B} \triangleq \left(\left({\displaystyle\oplus_{n=1}^{\mathbf{I}}}\left(\hat{\gamma}_n\Gamma_{n1}\!\right)\right)+\left({\displaystyle\oplus_{n=1}^{\mathbf{I}}}\left((1-\hat{\gamma}_n)\Gamma_{n0}\right)\!\right)\right)\left(Q'\otimes\mathbb{I}_{n_x}\right).
\eeq
\begin{align}
\mathcal{O}_n(M,k)&\triangleq\pi_n(k)\left(GG^*+\hat{\gamma}_nM_nHH^*M_n^*\right),\\
\mathcal{O}(M,k)&\triangleq\left[\mathcal{O}_n(M,k)\right]_{n=1}^{\mathbf{I}}.
\end{align}
\end{subequations}
\end{prop}
\begin{proof}
See Appendix.
\end{proof}
\noindent Define {\small $\mathcal{C}\!\triangleq\!Q'\!\otimes\!\mathbb{I}_{n_x^2}.$ }
Then, the matrix forms of \eqref{mathcalV} and \eqref{mathcalJ} can be written respectively as
\beq
\hat{\varphi}\left(\mathcal{V}(\mathbf{S})\right)=\tilde{\mathbf{\Lambda}}\hat{\varphi}\left((\mathbf{S})\right),\quad \hat{\varphi}\left(\mathcal{J}(\mathbf{S})\right)=\tilde{\mathbf{\Lambda}}^*\hat{\varphi}\left(\mathbf{S}\right),
\eeq
\noindent where
\begin{align*}
\tilde{\mathbf{\Lambda}}\triangleq&\Big[{\displaystyle\oplus_{m=1}^{\mathbf{I}}} \Big(\hat{\gamma}_m\left(\bar{\Gamma}_{m1}\otimes \Gamma_{m1}\right)\Big)+\nonumber\\
&{\displaystyle\oplus_{m=1}^{\mathbf{I}}} \Big((1-\hat{\gamma}_m)\left(\bar{\Gamma}_{m0}\otimes \Gamma_{m0}\Big)\right)\Big]\mathcal{C}.
\end{align*}
\noindent For all {\small $ \mathbf{S}\!=\!\bmat S_m\emat_{m=1}^{\mathbf{I}}\!\in\!\FF^{\mathbf{I}n_x\!\times \!n_x}$}, from \eqref{mathcalV} and \eqref{mathcalJ}, together with {\it Remark} \ref{remrhov}, it follows that {\small $\hat{\varphi}\!\left(\mathcal{V}\left(\mathbf{S}\right)\!\right)\!=\!\tilde{\mathbf{\Lambda}}\hat{\varphi}\left(\mathbf{S}\right)$,  $\hat{\varphi}\!\left(\mathcal{J}\left(\mathbf{S}\right)\!\right)\!=\!\tilde{\mathbf{\Lambda}}^*\!\hat{\varphi}\!\left(\mathbf{S}\right)$}. Thus, we have that {\small $\rho(\mathcal{V})\!=\!\rho(\!\mathcal{J}\!)\!=\!\rho(\!\tilde{\mathbf{\Lambda}}\!)$}.
In the following, we introduce the definition of mean square  detectability with respect to the sensing channel.
\begin{mydef}\label{msddef}
The system described by \eqref{mathcalG1} is {\it mean square detectable with respect to the sensing channel} if there exists a mode-dependent filtering gain {\small $M\!=\![M_n]_{n=1}^{\mathbf{I}}$},  such that \mbox{\small $\rho(\mathcal{V})\!<\!1$}, with {\small $\mathcal{V}$} defined as in \eqref{mathcalV}, and with \mbox{\small $\Gamma_{n1}\!=\!A\!+\!M_nL,$} \mbox{\small $\Gamma_{n0}\!=\!A$}.
\end{mydef}
\begin{rem}\label{remrhomss}
Applying the results presented in \cite[\it Section 3.4.2]{COSTA2005}, to the operator {\small $\mathcal{V}$} (with {\small $\mathcal{V}$} defined by \eqref{mathcalV}, and with {\small $\Gamma_{n1}\!=\!A\!+\!M_nL,\;\Gamma_{n0}\!=\!A$}), it follows that {\small $\rho(\mathcal{V})\!\!<\!\!1$} implies mean square stability of the error system \eqref{nexterrordynamics}.
\end{rem}
\subsection{The Filtering CARE}
\noindent In this subsection, 
we compute {\it the optimal mode-dependent filtering gain}. The performance index optimized by this filter is :
\beq
J_f=\limsup_{t\to\infty} \frac{1}{t}\EE\left[\sum_{k=0}^t\left(\te_k\te^*_k\right)|\mathcal{F}^k\right].
\eeq
\noindent By using a technical approach based on dynamic programming, it is immediate to see that the solution of the optimal infinite horizon filtering problem can be obtained from the following CARE.
We set for any {\small $Y\in\HH^{\mathbf{I}n_x,*}$}
\begin{align*}
\tilde{\mathcal{A}}_n\left(Y\right)&\triangleq A\mathcal{D}_n(Y)A^* + \pi_n^{\infty}GG^{*},\ 
\tilde{\mathcal{C}}_n\left(Y\right)\triangleq\!\hat{\gamma}_n^{\frac{1}{2}}\!A \mathcal{D}_n(Y) L^*,\\
\tilde{\mathcal{R}}_n\left(Y\right)&\triangleq \pi_n^{\infty}HH^* + L \mathcal{D}_n(Y)L^*.
\end{align*}
\noindent Given the set 
\beqs
\LL\!\triangleq\!\{\!Y\!\in\!\HH^{\mathbf{I}n_x,*};\ \tilde{\mathcal{R}}_n(Y) 	\text{ non-singular }
\forall n \in\!\mathbb{S}_{\eta}\},
\eeqs
\noindent we define for any {\small $Y\!\in\!\LL$} 
\begin{align}
\mathcal{Y}_n(Y)&\triangleq \tilde{\mathcal{A}}_n(Y) - \tilde{\mathcal{C}}_n(Y)\tilde{\mathcal{R}}_n(Y)^{- 1}\tilde{\mathcal{C}}_n^*(Y),\;n\in\mathbb{S}_{\eta};\\
\mathcal{Y}(\cdot)&\triangleq [\mathcal{Y}_n(\cdot)]_{n=1}^{\mathbf{I}}.\nonumber
\end{align}
We call {\it Filtering CARE} the set of equations
\beq\label{FiltCARE0}
Y_n = \mathcal{Y}_n(Y),\,n \in \mathbb{S}_{\eta}.
\eeq
\indent
In the following, we show that the optimal solution of the Filtering CARE \eqref{FiltCARE0} can be obtained through a linear matrix inequality (LMI) approach.
\begin{pb}\label{mypb}
Consider the following optimization problem
\begin{subequations}\label{optpbriccati}
\begin{align}
\max\,\mathrm{tr}\left(\sum\limits_{n=1}^{\mathbf{I}} Y_n\right)
\end{align}
\text{subject to}
\begin{align}
\bmat 
-Y_n+\tilde{\mathcal{A}}_n(Y) && \tilde{\mathcal{C}}_n(Y)\\
\tilde{\mathcal{C}}^*_n(Y) &&  \tilde{\mathcal{R}}_n(Y)
\emat \succeq 0,
\end{align}
\beq
\tilde{\mathcal{R}}_n(Y)\!\succ\!0 ,\;Y\!\in\!\HH^{\mathbf{I}n_x,*},\; n\!\in\!\mathbb{S}_{\eta}.
\eeq
\end{subequations}
\end{pb}
Given the set  \vspace*{-2mm}
\beqs
\mathbb{M} \triangleq \{ Y \in\LL;\ \tilde{\mathcal{R}}(Y) \succ 0 \ \text{and } - Y + \mathcal{Y}(Y)\succeq 0\}
\eeqs
\noindent
we present the following theorem.
\begin{thm}[{\it Solution of Problem  \ref{mypb}}]\label{thmmaximal}
Assume that  \eqref{mathcalG1} is mean square detectable according to {\it Definition} \ref{msddef}. Then, there exists \mbox{\small $Y^+\!\in \!\MM$}, satisfying \eqref{FiltCARE0}, such that \mbox{\small $Y^+ \succeq Y$}, for all \mbox{\small $Y\in \MM$}, if and only if there exists a solution \mbox{\small $\widehat{Y}$} for the above convex programming problem. Moreover, \mbox{\small$\hY = Y^+$}.
\end{thm}
\begin{proof}
See Appendix.
\end{proof}
The optimal mode-dependent filtering gain is:
\begin{align*}
&M_n=\mathcal{M}_n(Y
)\triangleq -A \mathcal{D}_n(Y) L^*\left(\pi_n^{\infty}HH^*+L \mathcal{D}_n(Y) L^*\right)^{-1},\nonumber\\
& n\in\mathbb{S}_{\eta},
\end{align*}
\noindent where {\small $Y\in\MM$} is the maximal solution of  \eqref{FiltCARE0}, i.e. it is the solution of Problem \ref{mypb}, and the {\it optimal performance index} achieved by the filter is:
\beq
J_f^{*}=\sum\limits_{m=1}^{\mathbf{I}} \pi^{\infty}_m\mathrm{tr}(Y_m).
\eeq
\begin{mydef}[{\it Mean square stabilizing solution of \eqref{FiltCARE0}}]\label{def:mssol}
We say that {\small $Y\!\!\in\!\!\LL$} is the mean square stabilizing solution for the Filtering CARE if it satisfies \eqref{FiltCARE0} and {\small $\rho(\mathcal{V})\!\!<\!\!1$}, with {\small $\Gamma_{n1}\!=A\!+\!\mathcal{M}_n(Y)L,\,\Gamma_{n0}\!=A,\,n\!\in\!\mathbb{S}_{\eta}$}; i.e. {\small $\mathcal{M}_n(Y)$} stabilizes the error system \eqref{nexterrordynamics} in the mean square sense.
\end{mydef}
\noindent We present the  connection between the maximal solution and the mean square stabilizing 
solution for the Filtering CARE \eqref{FiltCARE0} in the next theorem.
\begin{thm}[{\it Mean square stabilizing solution of \eqref{FiltCARE0}}]\label{lemuniquestab} 
There exists at most one {\it mean square stabilizing solution} for the Filtering CARE, which will coincide with the maximal solution in {\small $\MM$}, that is the solution of the above convex programming problem. 
\end{thm}
\begin{proof}
See Appendix.
\end{proof}

\noindent Clearly the necessary condition for the existence of the mean square stabilizing solution of the Filtering CARE is the mean square detectability of system \eqref{mathcalG1}.
\section{The separation principle}\label{sec:sep_princ}
\noindent
Consider the optimal output feedback controller \eqref{K1}, with optimal matrices
\begin{align*}
\widehat{A}(\nu_{\theta_k},\theta_{k-1},\gamma_{\eta_k}) &= A + \nu_{\theta_k}BF_{\theta_{k-1}} + \gamma_{\eta_k}M_{\eta_k}L,\\\widehat{B}_{\eta_k} &= -M_{\eta_k},\ \widehat{C}_{\theta_{k-1}} = F_{\theta_{k-1}}.
\end{align*}
\noindent Then \eqref{K1} coincides with \eqref{tildemathcalG}, and the dynamics of {\small $x_{k+1}$} becomes:
\begin{align*}
x_{k\!+\!1}=\left(A + \nu_{\theta_{k}}BF_{\theta_{k - 1}}\right)x_{k} - \nu_{\theta_{k}}BF_{\theta_{k - 1}}\te_{k} + Gw_{k}.
\end{align*}
\indent Recalling the error dynamics in \eqref{nexterrordynamics},
 the closed-loop system dynamics is given by:
\beq\label{mathcalGcl}
\tilde{\mathcal{G}}_{cl}:\,\mathbf{\mathcal{E}}_{k+1}=
\mathbf{\Gamma}\left(\nu_{\theta_k},\theta_{k-1},\gamma_{\eta_k}\right)\mathbf{\mathcal{E}}_{k} + 
\mathbf{\Sigma}(\gamma_{\eta_k}\!) w_{k},
\eeq
\noindent with
\begin{align}
\mathbf{\mathcal{E}}_k\!\triangleq\! \bmat x_{k}\\\te_{k}\emat,\quad\;
\mathbf{\Sigma}\left(\gamma_{\eta_{k}}\right)\!
\triangleq\!\bmat G \\G\!+\!\gamma_{\eta_{k}}M_{\eta_{k}}H\emat ,\qquad\qquad\qquad\qquad\\
\mathbf{\Gamma}\!\left(\nu_{\theta_k},\theta_{k-1},\gamma_{\eta_k}\right)\!\triangleq\!\bmat
\left(A\!+\!\nu_{\theta_{k}}BF_{\theta_{k-1}}\right) && -\nu_{\theta_{k}}BF_{\theta_{k-1}}\\
O_{n_x} && \left(A+\gamma_{\eta_{k}}M_{\eta_{k}}L\right)
\emat .\label{eq:GAMMA}
\end{align}
\noindent 
In this section, we present the separation principle, as the main result of this paper. 
%
\noindent
\begin{thm}\label{thmsep}
Given Markov Jump linear system \eqref{mathcalG1}, and the Luenberger like observer \eqref{tildemathcalG}, the dynamics \eqref{mathcalGcl} can be made mean square stable if and only if system \eqref{mathcalG1} is mean square detectable according to Definition \ref{msddef}, and mean square stabilizable with one time step delay in the observation of  actuation channel mode according to Definition  \ref{Def:MSstab}.
\end{thm}
\begin{proof}
See Appendix.
\end{proof}
\begin{rem}
Differently from \cite{COSTA2005}, the matrix {\small $\mathbf{\Gamma}\left(\nu_{\theta_k},\theta_{k-1},\gamma_{\eta_k}\right)$} contains the Markov jumps not only of the MC {\small $\eta$} (sensing channel), but of the MC {\small $\theta$} (actuation channel) too. Moreover, we consider the actuation delay that affects the MC {\small $\{\theta_k\}_{k\in\NN}$}. Finally, is \eqref{eq:GAMMA} is an upper triangular block matrix, i.e. the error dynamics (driven by the MC {\small $\eta$}) does not depend on the state dynamics (driven by the MC {\small $\theta$}).
\end{rem} 
\section{NUMERICAL EXAMPLE}\label{sec:ex}
\noindent Consider the inverted pendulum on a cart as in \cite{franklin2009feedback}. 
%
The state variables of the plant are the cart position coordinate {\small $\mathrm{x}$} 
and the pendulum angle from vertical {\small $\phi$}, together with respective first 
derivatives. 
We aim to design a controller that stabilizes the 
pendulum in up-right position, corresponding to unstable equilibrium point 
{\small $\mathrm{x}^{\star}\!\!=\!\!0\,$}m, {\small $\phi^{\star}\!\!=\!\!0\,$}rad. The system state is
defined by {\small $x\!=\!\begin{bmatrix}
\delta\mathrm{x}, \delta\dot{\mathrm{x}}, \delta\phi, \delta\dot{\phi}
\end{bmatrix}'$}, where 
{\small $\delta\mathrm{x}(t)\!=\!\mathrm{x}(t)\!-\!\mathrm{x}^{\star}$}, and
{\small $\delta\phi(t)\!=\!\phi(t)\!-\!\phi^{\star}$}. The initial state of the plant is
{\small $x_0\!=\!\begin{bmatrix} 0, 0, \frac{\pi}{10}, 0 \end{bmatrix}'\!$}, while the initial state of the observer is {\small $\tilde{x}_0\!=\!\bmat 1,0,(11\pi/100),0 \emat '$}. The optimal Markov jump output-feedback controller \eqref{K1} has been applied to the  discrete time linear model derived from the continuous time nonlinear model, by linearization.
The state space model of the system is linearized around the unstable 
equilibrium point and discretized with sampling period 
{\small $\mathrm{T}_{\mathrm{s}}\!=\!0.01\,$}s:
%
\vspace*{-1mm}
\begin{small}
\begin{equation*}
A\!=\!\begin{bmatrix}
1.000 &  0.010, &  0.000 &  0.000\\
0.000 &   0.998 &  0.027 &  0.000\\
0.000 &  0.000 &  1.002 &  0.010\\
0.000 &  -0.005 &  0.312 &   1.002\end{bmatrix}\!,\,
B\!=10^{-1}*\!\begin{bmatrix}
	0.00091 \\ 0.182 \\ 0.0023 \\ 0.474
\end{bmatrix}\!.
\end{equation*}
\end{small}The weighting matrices in {\small $z_k$} are {\small $C^*C\!\!\!\!=\!\!\!\!\bigoplus (1000,0.1,10000,0.1)$, $D\!\!\!=\!\!\!1$}, while matrices {\small $H$ } and {\small $G$ } are such that {\small
$
HH^*\!\!=\!\!\mathbb{I}_{n_x}\!\!\succ\!\!0,\; GG^*=\bmat \II_2 & \II_2\\ \II_2 & \II_2
\emat\!\!\succeq\!\!0.
$} 
The process noise is characterized by the covariance matrix
{\small $\EE[w_kw_k^*]\!\!=\!\!\tilde{\alpha}_w\mathbb{I}_{n_w}$}\footnote{We consider the noise covariance matrix as a positive scalar less than {\small $1$, $\tilde{\alpha}_w$}, multiplying the identity matrix. Indeed, the results shown in the previous section can be applied without any loss of generality}, with {\small $\tilde{\alpha}_w=0.0002$}.
The state matrix {\small $A$} is unstable, since it has an eigenvalue {\small $1.058$}, but
it is easy to verify that {\small $D^*D\!\!\succ\!\!0$}, {\small $C^*C\!\!\succeq\!\!0$}, 
the pair {\small $\left(A,B\right)$} is controllable, while
the pair {\small $\left(A,L\right)$},  is observable, 
so the closed-loop system 
is asymptotically stable, if {\small $\nu_k\!\!=\!\!1$ and $\gamma_k\!\!=\!\!1$} {\small $\forall k$}. Moreover, the necessary conditions for the {\it existence of the mean square stabilizing solution for the Control and Filtering CARE, are satisfied.} 
The double sided packet loss is described by Markov channels with {\bf TPMs} in {\small $\RR^{12\times 12}$}, \footnote{The symbol  "{\small $\cdots$}" in the TPMs stands for elements that are approximately equal to zero, i.e. elements with the first four decimal numbers equal to zero.}:
\begin{small}
\beqs
\resizebox{0.9\hsize}{!}{%
$
P\!=\!Q\!=\!\bmat     0  &  2\cdot 10^{-4}  & \cdots  &  1\cdot 10^{-4}  &  1\cdot 10^{-4}  &  1\cdot 10^{-4}  &  0.0071  &  0.9922\\
           0  &  2\cdot 10^{-4}  &  \cdots  &  1\cdot 10^{-4}  &  1\cdot 10^{-4}  &  1\cdot 10^{-4}  &  0.0070  &  0.9923\\
           0  &  2\cdot 10^{-4}  &  \cdots  &  1\cdot 10^{-4}  &  1\cdot 10^{-4}  &  1\cdot 10^{-4}  &  0.0070  &  0.9924\\
           0  &  2\cdot 10^{-4}  &  \cdots  &  1\cdot 10^{-4}  &  1\cdot 10^{-4}  &  1\cdot 10^{-4}  &  0.0069  &  0.9924\\
           0  &  2\cdot 10^{-4}  &  \cdots  &  1\cdot 10^{-4}  &  1\cdot 10^{-4}  &  1\cdot 10^{-4}  &  0.0069  &  0.9924\\
           0  &  2\cdot 10^{-4}  &  \cdots  &  1\cdot 10^{-4}  &  1\cdot 10^{-4}  &  1\cdot 10^{-4}  &  0.0069  &  0.9924\\
           0  &  2\cdot 10^{-4}  &  \cdots  &  1\cdot 10^{-4}  &  1\cdot 10^{-4}  &  1\cdot 10^{-4}  &  0.0069  &  0.9924\\
           0  &  2\cdot 10^{-4}  &  \cdots  &  1\cdot 10^{-4}  &  1\cdot 10^{-4}  &  1\cdot 10^{-4}  &  0.0069  &  0.9924\\
           0  &  2\cdot 10^{-4}  &  \cdots  &  1\cdot 10^{-4}  &  1\cdot 10^{-4}  &  1\cdot 10^{-4}  &  0.0069  &  0.9924\\
           0  &  2\cdot 10^{-4}  &  \cdots  &  1\cdot 10^{-4}  &  1\cdot 10^{-4}  &  1\cdot 10^{-4}  &  0.0069  &  0.9924\\
           0  &  2\cdot 10^{-4}  &  \cdots  &  1\cdot 10^{-4}  &  1\cdot 10^{-4}  &  1\cdot 10^{-4}  &  0.0068  &  0.9925\\
           0  &  2\cdot 10^{-4}  &  \cdots  &  1\cdot 10^{-4}               & 1\cdot 10^{-4}  &  1\cdot 10^{-4}  &  0.0063            &  0.9931\emat$
}
\eeqs
\end{small}
and packet losses probability vectors
\begin{small}
\beqs
\resizebox{0.9\hsize}{!}{%
$
\begin{aligned}
\hat{\nu}\!\!&=\!\!\hat{\gamma}\\
&=\![0,    0.02, 0.15,    0.25,    0.35,    0.45, 0.55,    0.65,    0.75,    0.86,    0.99,    1].
\end{aligned}
$
}
\eeqs
\end{small}These channels are obtained by following the systematic procedure in \cite{Lun2020ontheimpact} accounting for path loss, shadow fading, transmission power control and interference. The partitioning of the SNIR range is based on the values of PEP, so that to each SNIR threshold corresponds a specific value of PEP.
\begin{figure}
	\centering	\includegraphics[width=0.5\textwidth]{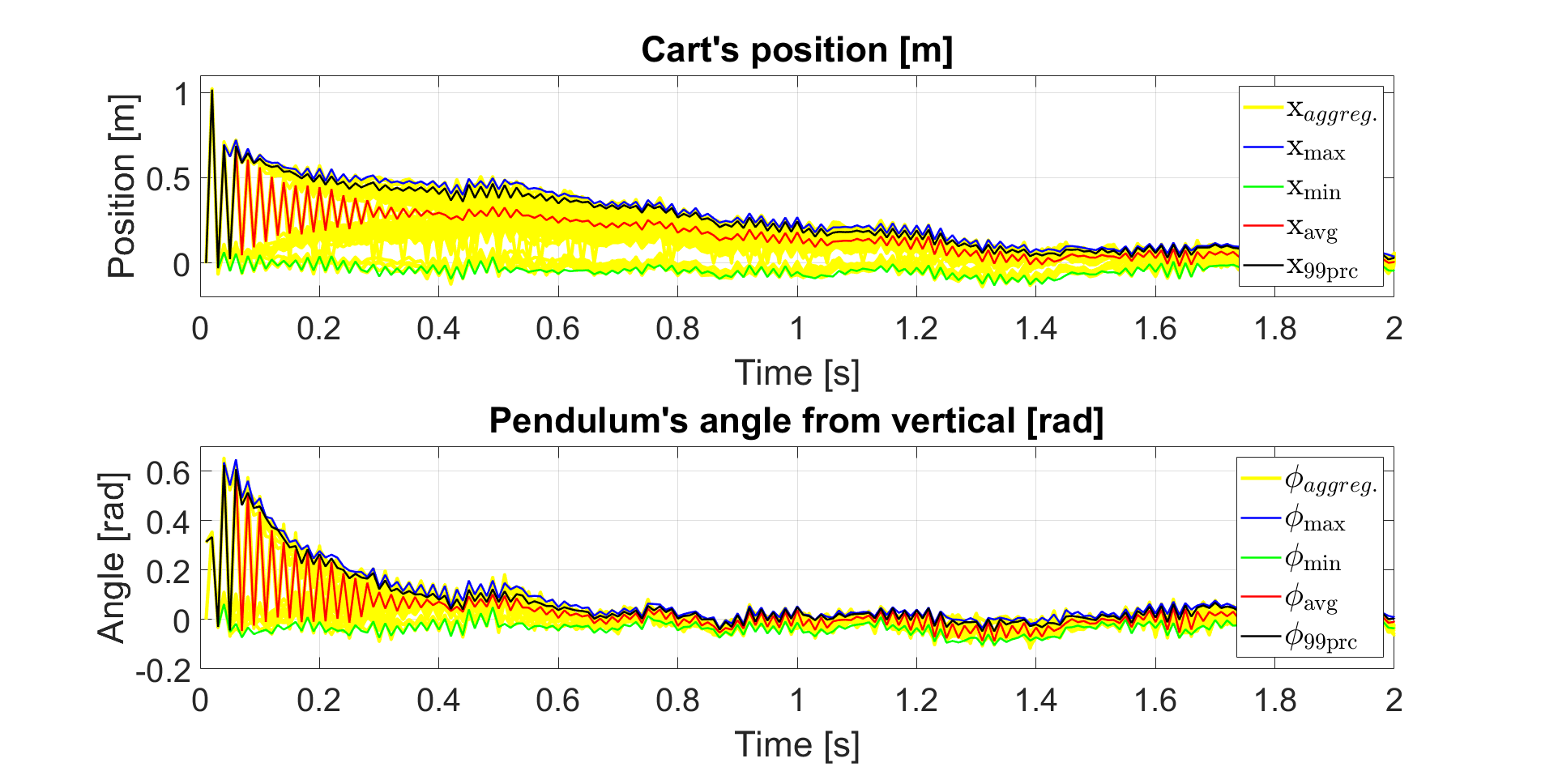}\\
	\vspace{-4mm}	
	\caption{Traces of the system states, generated in the case with noise.}
	\label{fig:Markov-noisestates}
\end{figure}
\begin{figure}
	\centering	\includegraphics[width=0.5\textwidth]{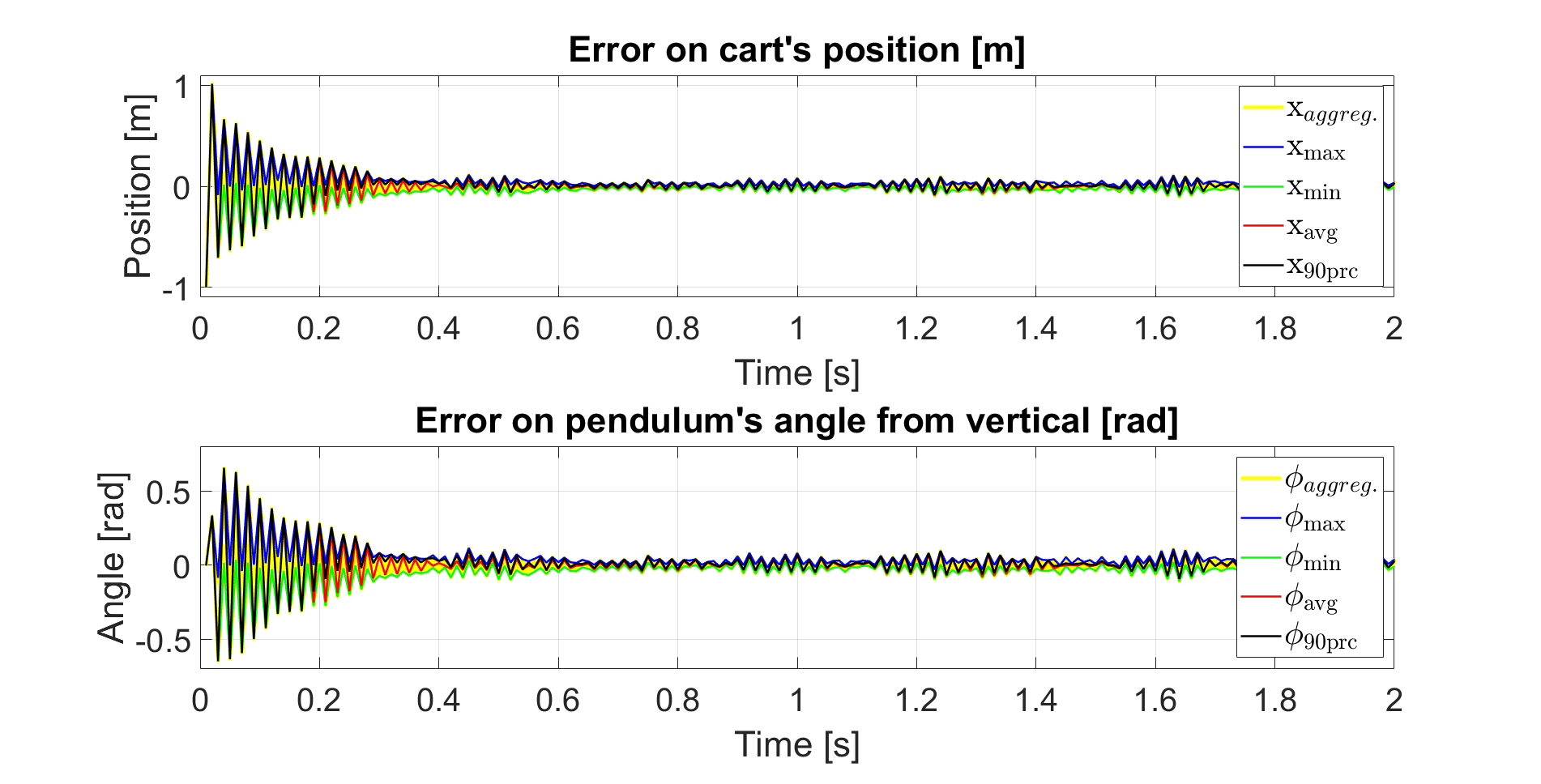}\\
	\vspace{-4mm}	
	\caption{Traces of the error, generated in the case with noise.}
	\label{fig:Markov-noiseerror}
\end{figure}
\begin{figure}
	\centering	\includegraphics[width=0.5\textwidth]{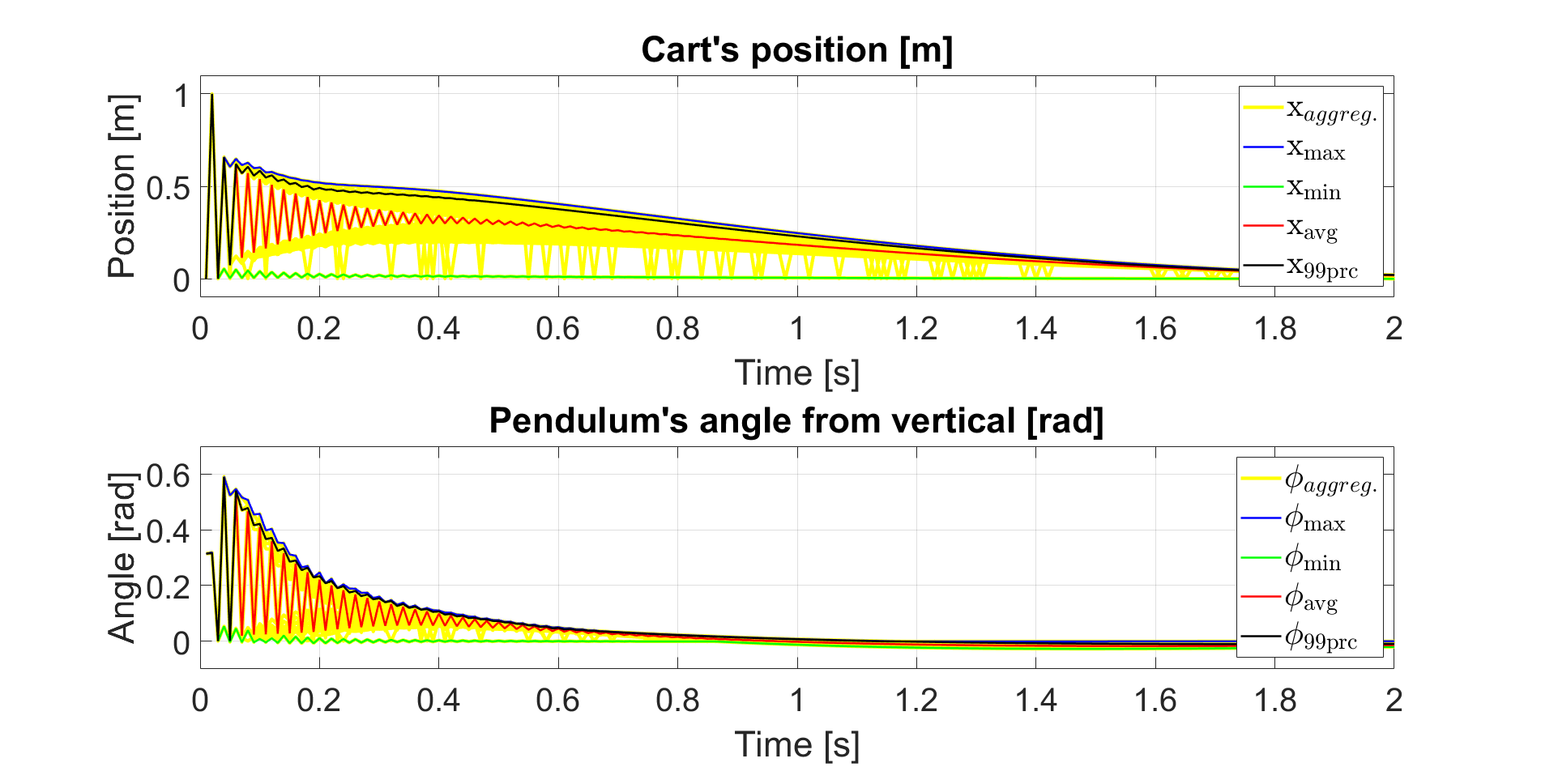}\\
	\vspace{-4mm}	
	\caption{Traces of the system state, generated in the noiseless case.}	\label{fig:Markov-nlstates}
\end{figure}
\begin{figure}
	\centering	\includegraphics[width=0.5\textwidth]{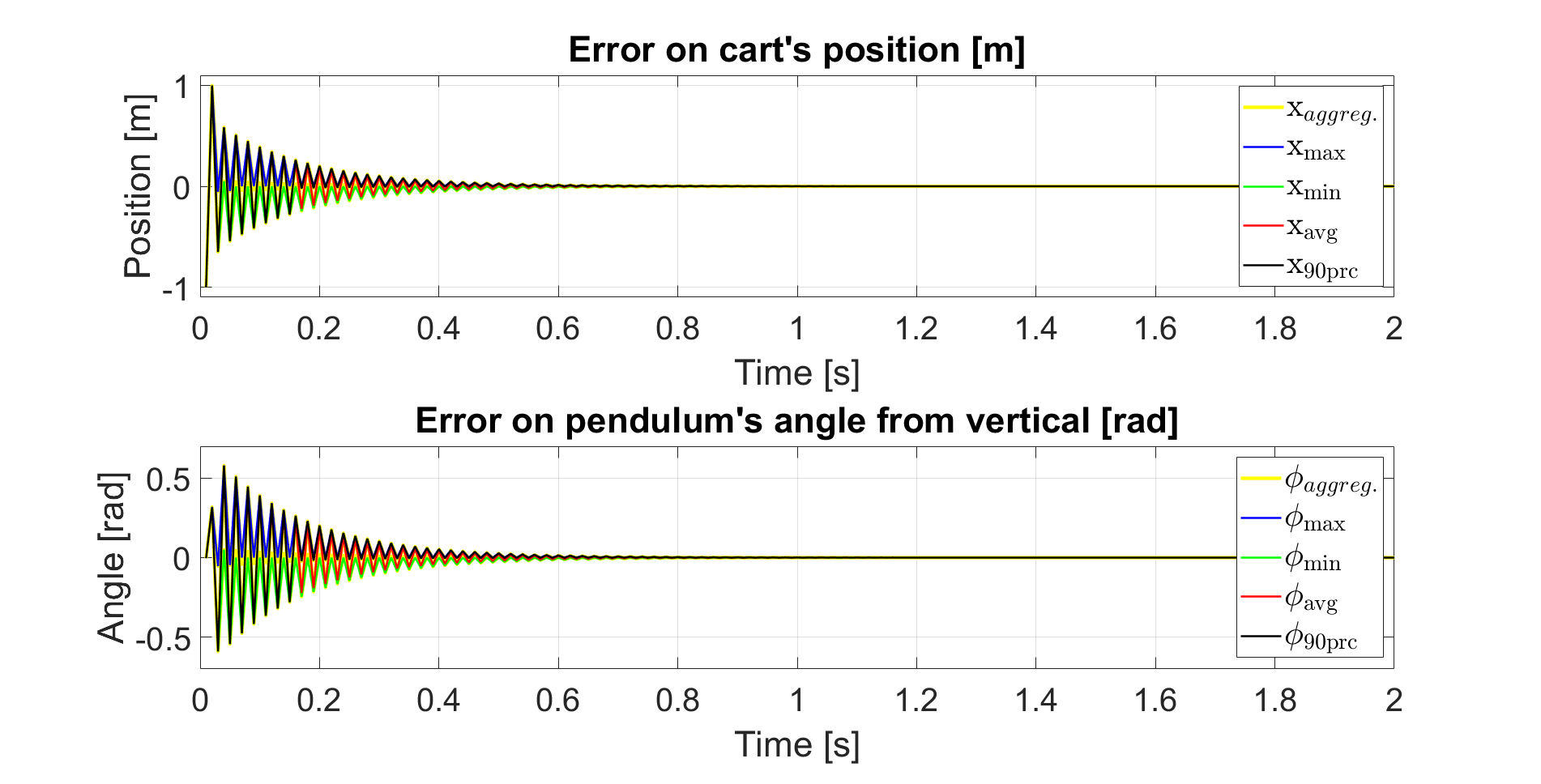}\\
	\vspace{-4mm}	
	\caption{Traces of the error, generated in the noiseless case.}	\label{fig:Markov-nlerror}
\end{figure}In Fig. \ref{fig:Markov-noisestates}-\ref{fig:Markov-noiseerror}, we show the traces of the system states, and of the error components, respectively, in  presence of noise. Comparing figures \ref{fig:Markov-noisestates}-\ref{fig:Markov-noiseerror} with \ref{fig:Markov-nlstates}-\ref{fig:Markov-nlerror}, it may be noted that all the aforementioned trajectories show a convergent behaviour. Indeed the closed-loop system is mean square stable. Moreover, in the noiseless case the traces of the state and of the error converge to zero as we expected. The simulation results with this practical example fully validate the provided theory.
\section{Conclusion}\label{sec:Conclusions}
In this paper we considered the application of Markov jump linear systems to wireless networked control scenarios. We generalize the results of \cite{Lun2019stabilizability} to double-sided packet loss as one of the contributions.
Moreover, we design the optimal output-feedback controller, that 
can be obtained by solving separately the optimal control problem, and the  optimal filtering problem, using two kinds of coupled algebraic Riccati equations: one associated to the optimal filtering problem and the other one associated to the optimal control problem.
\bibliography{IEEEabrv,mybib1}
\appendix
%
\begin{proof}[Proof of Proposition \ref{propmeanerror}]
As far as the expression of {\small $\tm_n(k+1)$} is concerned, applying  the definition in \eqref{nextmnkdef}, the expression of \eqref{nexterrordynamics}, from the error dynamics and from the assumption that {\small $\EE\left[w_k\right]=0$}, we obtain the expression of {\small $
\tm_n(k+1)$} in  \eqref{nexttmkexp}. \\\noindent
As far as the expression of {\small $\tY_n(k+1)$} is concerned, applying  the definition in \eqref{nextYnkdef}, and the expression of the error dynamics in \eqref{nexterrordynamics}, the assumption {\small $GH^*=0$}, and the definition of $\mathcal{V}_n(\cdot)$, in \eqref{mathcalV}, one can easily obtain the expression of {\small $\tY_n(k+1)$} in \eqref{nexttmkexp}.\\\noindent The proof of the proposition is complete.
\end{proof}
\noindent In the following, we present instrumental results for the proof of separation principle. 
\begin{lem}\label{lemequations}
Suppose that $Y \in \LL$ and for some $\hM = [\hM_n]_{n=1}^{\mathbf{I}} \in\FF^{\mathbf{I} n_x\times n_y}$, $\hY \in \HH^{n_x,*}$, satisfies for $n \in \mathbb{S}_{\eta}$
\begin{align}\label{eqcondition1}
&\hY_n\!-\hat{\gamma}_n\!\left(\!A\!+\!\hM_n L\!\right)\!\mathcal{D}_n(\hY)\!\left(\!A+\hM_n L\!\right)^*\!\!-\nonumber\\
& (\!1\!-\!\hat{\gamma}_n)A\mathcal{D}_n(\hY)A^*\!= \mathcal{O}_n\!\left(\hM\right),\nonumber\\
&\mathcal{O}_n(\hM)\!\triangleq \!\pi_n^{\infty}\!\left(GG^*\!+\!\hat{\gamma}_n \hM_nHH^*\hM_n^*\!\right),
\end{align}
\noindent then, for $n \in \mathbb{S}_{\eta}$,
\begin{align}\label{item1}
&(\hY_n\!-\!Y_n)\!-\!\hat{\gamma}_n\!\left(A\!+\!\hM _n L\right)\!\mathcal{D}_n(\hY \!-\!Y)\!\left(A\!+\hM _n L\right)^*
\!-\nonumber\\
&(1\!-\!\hat{\gamma}_n)A\mathcal{D}_n(\hY\!-\!Y)A^*\!=\nonumber\\
&\mathcal{Y}_n\left(Y\right)\!-\!Y_n\!+\!\hat{\gamma}_n\left(\hM _n\!-\!\mathcal{M}_n(Y)\right)\!\tilde{\mathcal{R}}_n(Y)\!\left(\hM_n\!-\!\mathcal{M}_n(Y)\right)^*;\qquad\qquad\qquad
\end{align}
\noindent moreover, if $\hY\in\mathbb{L}$,  for $n \in \mathbb{S}_{\eta}$,
\begin{align}\label{item2}
&\left(\hY_n\!-\!Y_n\right)\!-\!\hat{\gamma}_n\!\left(A\!+\!\mathcal{M} _n(\hY) L\!\right)\!\mathcal{D}_n(\hY\!-\!Y)\!\left(A\!+\!\mathcal{M} _n(\hY) L\right)^*\!\!-\nonumber\\
&(1\!-\!\hat{\gamma}_n)A\mathcal{D}_n(\hY\!-\!Y)A^*=\nonumber\\&\hat{\gamma}_n\!\left(\mathcal{M}_n(\hY)\!-\!\mathcal{M}_n(Y)\right)\!\tilde{\mathcal{R}}_n(Y)\!\left(\mathcal{M}_n(\hY)\!-\!\mathcal{M}_n(Y)\right)^*\nonumber\\
&+\hat{\gamma}_n\!\left(\hM_n\!-\!\mathcal{M}_n(\hY)\right)\!\tilde{\mathcal{R}}_n(\hY)\!\left(\hM_n\!-\!\mathcal{M}_n(\hY)\right)^*\!+\!\mathcal{Y}_n(Y)\!-\!Y_n;
\end{align}
\item furthermore, if $\hX\!\in \! \HH^{\mathbf{I}n_x,*}$ and satisfies, for $n \in \mathbb{S}_{\eta}$
\begin{align}\label{eqcondition3}
&\hX_n\!-\!\hat{\gamma}_n\!\left(A\!+\!\mathcal{M}_n(\hY)L\right)\!\mathcal{D}_n(\hX)\!\left(A\!+\!\mathcal{M}_n(\hY)L\right)^*\!-\nonumber\\
&(1\!-\!\hat{\gamma}_n)A\mathcal{D}_n(\hX)A^* = \mathcal{O}_n(\mathcal{M}(\hY)),
\end{align}
\noindent for $n \in\!\mathbb{S}_{\eta}$ then,
\begin{align}\label{item3}
&\left(\hY_n\!-\!\hX_n\right)\!-\!\hat{\gamma}_n\left(A\!+\!\mathcal{M}_n(\hY)L\right)\!\mathcal{D}_n(\hY\!-\!\hX)\!\left(A\!+\!\mathcal{M}_n(\hY)L\right)^*\!\!\nonumber\\
&-(1\!-\!\hat{\gamma}_n)A \mathcal{D}_n(\hY\!-\!\hX)A^* =\nonumber\\
&=\!\hat{\gamma}_n\!\left(\hM_n\!-\!\mathcal{M}_n(\hY)\right)\!\tilde{\mathcal{R}}_n(\hY)\!\left(\hM_n\!-\!\mathcal{M}_n(\hY)\right)^*.
\end{align}
\end{lem}
\begin{proof}
Let us show that \eqref{item1} holds. Consider the left-hand side of \eqref{item1} for $n\!\in\!\mathbb{S}_{\eta}$, applying \eqref{eqcondition1} and the definitions of $\mathcal{M}_n(Y)$, the right-hand side of \eqref{item1} is easily obtained. To show that \eqref{item2} holds for $n\in\mathbb{S}_{\eta}$, consider the left-hand side of equality \eqref{item2}, applying the definitions of $\tilde{\mathcal{R}}_n(\hY)$, $\tilde{\mathcal{R}}_n(Y)$, $\mathcal{M}_n(Y)$, and \eqref{eqcondition1}, equality \eqref{item2} holds. Let us show that equality \eqref{item3} holds. Consider the left-hand side of equality \eqref{item3}, applying the definition of $\mathcal{M}_n(\hY)$ and $\tilde{\mathcal{R}}_n(\hY)$, and \eqref{eqcondition3}, the reader can easily obtain equality \eqref{item3}. \\
\noindent The proof of the lemma is complete.
\end{proof}
\begin{lem}\label{lemmathvmathvbar}
Let $\mathcal{V}$ and $\tilde{\mathcal{V}}$, be defined as in \eqref{mathcalV} and \eqref{barmathcalV}, respectively.
Suppose that $\rho(\mathcal{V})<1$ and for some \mbox{$Y\in\HH^{\mathbf{I}n_x,+}$} and $\delta\!>\!0$,
\beq\label{Yminusbarmathv}
Y_n\!-\!\tilde{\mathcal{V}}_n(Y)\!\succeq\!\delta\hat{\gamma}_n\!\left(K_n\!-\!M_n\right)\left(K_n\!-\!M_n\right)^*,\, n\!\in\!\mathbb{S}_{\eta},
\eeq
\noindent with $\tilde{\mathcal{V}}$ defined in \eqref{barmathcalV}, with $\Lambda_{n1}\!\!=\!\!A\!+\!K_nL,\,\Lambda_{n0}\!\!=\!\!A$.\\
\noindent Then, $\rho(\tilde{\mathcal{V}})\!<\!1$.
\end{lem}
\begin{proof}
Set $\tilde{\mathcal{J}}\!=\!\tilde{\mathcal{V}}^*$. Note that for arbitrary $\epsilon\!>\!\!0$ and $V\!\!\in\!\HH^{\mathbf{I}n_x,+}$,
\begin{align}\label{zeroleqhatgamma}
0\preceq&\hat{\gamma}_n \left[\epsilon\left(A\!+\!M_nL\right)\!-\!\frac{1}{\epsilon}(K_n\!-\!M_n)L\right]^*V_n\times\nonumber\\&\left[\epsilon\left(A\!+\!M_nL\right)\!-\!\frac{1}{\epsilon}(K_n\!-\!M_n)L\right]\!+\nonumber\\
&+(1\!-\!\hat{\gamma}_n)\epsilon^2 A^*V_nA.
\end{align}
\noindent By applying the previous inequality, we get
\beqs
0\preceq \left(1+\epsilon^2 \right)\mathcal{J}_m(V)+\left(1+\frac{1}{\epsilon^2}\right)\mathcal{Q}_m(V),
\eeqs
where
$\mathcal{Q}(\cdot)\!\triangleq\![\mathcal{Q}_m(\cdot)]_{m=1}^{\mathbf{I}}$,
\begin{align*}
&\mathcal{Q}_m(V)\triangleq
{\displaystyle\sum_{n=1}^{\mathbf{I}}}\!q_{mn}\hat{\gamma}_n L^*(K_n\!-\!M_n)^*V_n (K_n\!-\!M_n)L,\nonumber\\
&\hat{\mathcal{Q}}(V)\!\!\triangleq\!\!\left(1+\frac{1}{\epsilon^2}\right)\mathcal{Q}(V).
\end{align*}
\noindent Moreover, we define 
\beqs
\hat{\mathcal{J}}(\cdot)\triangleq\left(1\!+\!\epsilon^2\right)\mathcal{J}(\cdot).
\eeqs
Since \mbox{
$\mathcal{J}\!\!=\!\!\mathcal{V}^*,$} and \mbox{$\rho(\mathcal{V})\!<\!1$} by hypothesis, we have that \mbox{$\rho(\mathcal{J})=\rho(\mathcal{V}) < 1$}. 
Therefore, we can choose $\epsilon\!>\!0$, such that \mbox{$\rho(\hat{\mathcal{J}})\!\!<\!\!1$}.
Let us define for \mbox{$t\!=\!0,\ldots,$} the sequences
\begin{align*}
U(t\!+\!1)&\triangleq\tilde{\mathcal{J}}(U(t)),\qquad\qquad\qquad
U(0)\succeq 0,\\
Z(t + 1)&\triangleq \hat{\mathcal{J}}(Z(t)) + \hat{\mathcal{Q}}(U(t)),\ Z(0)= U(0).
\end{align*}
\noindent At this point, we have to prove that
\beq\label{hatQlessinfty}
\sum\limits_{s=0}^{\infty}\|\hat{\mathcal{Q}}\left(U(s)\right)\|_1\!<\!\infty.
\eeq
\noindent Recalling the definition of norm, the properties of the trace operator, inequality \eqref{Yminusbarmathv}
 and the definition of inner product we obtain:
\begin{align}\label{eq:hatQnorm1}
\|\hat{\mathcal{Q}}\left(U(s)\right)\|_1\!\!=\!\!\sum_{m=1}^{\mathbf{I}} \|\hat{\mathcal{Q}}_m\left(U(s)\right)\|
\!\!\leq\!\!c_0\!\left\langle Y\!-\!\tilde{\mathcal{V}}(Y) ;U(s)\right\rangle,
\end{align} 
\noindent with \mbox{$c_0\triangleq(1/\delta)\left(1+1/\epsilon^2\right)\|L\|^2 \mathbf{I}$}. From \eqref{eq:hatQnorm1}, applying the properties of the inner product, and
taking the sum from $s\!=\!0$ to $\tau$, we get
\begin{align*}
\sum\limits_{s=0}^\tau \|\hat{\mathcal{Q}}\left(U(s)\right)\|_1\leq c_0 \left\langle Y;U(0)\right\rangle.
\end{align*}
\noindent Taking the limit for $\tau \to\infty$, we obtain that \eqref{hatQlessinfty} holds.
Following the same steps provided by \cite[\it Lemma A.8]{COSTA2005}, we can prove that
\begin{align}
0&\leq \sum_{t=0}^{\infty} \|\tilde{\mathcal{J}}^{t}(U(0))\|_1 = \sum_{t=0}^{\infty}\|U(t)\|_1\nonumber\\
&\leq \sum\limits_{t=0}^{\infty} \|Z(t)\|_1<\infty .
\end{align}
\noindent By \cite[\it Proposition 2.5]{COSTA2005}, $\rho(\tilde{\mathcal{J}})\!<\!\!1$. Therefore, \mbox{$\rho(\tilde{\mathcal{V}}) =
\rho(\tilde{\mathcal{J}}
)<1$}.\\\noindent
The proof of the lemma is complete.
\end{proof}
\begin{lem}\label{lemmaximal}
Assume that system \eqref{mathcalG1} is mean square detectable according to Definition \ref{msddef}.\\\noindent 
Then, there exists $Y^+ \in\MM,\,Y^+\succeq Y$, for any $Y\in\MM$, satisfying 
\eqref{FiltCARE0}.
\end{lem}
\begin{proof}
Consider an arbitrary \mbox{$Y\in\MM$}. We want to show that there exists a decreasing sequence \mbox{$\{Y^l\}_{l=0}^{\infty}$}, \mbox{$Y^l \in\HH^{\mathbf{I}n_x,+}$}, satisfying equation \eqref{OnMl}, for \mbox{$l=0,1,\ldots$}, with \mbox{$M^l \triangleq [M_n^l]_{n=1}^{\mathbf{I}}$}
\begin{align}\label{OnMl}
&Y_n^l\! -\! \mathcal{V}_n^l(Y^l)\!=\!\mathcal{O}_n(M^l),\quad \text{with }\mathcal{V}^l(\cdot)\!\triangleq \![\mathcal{V}_n^l(\cdot)]_{n=1}^{\mathbf{I}},\nonumber\\
&\mathcal{V}_n^l(\cdot)\!\triangleq \!\hat{\gamma}_n A_{n1}^l\mathcal{D}_n(Y^l)A_{n1}^{l*}\!+\!(1\!-\!\hat{\gamma}_n)A_{n0}^l\mathcal{D}_n(Y^l)A_{n0}^{l*},\nonumber\\ 
& M_n^l\! \triangleq\! \mathcal{M}_n(Y^{l\!-\!1}),\ A_{n1}^l \triangleq A\!+\!M_n^lL,\; A_{n0}^l \triangleq  A,
\;n \in \mathbb{S}_{\eta};
\end{align}
\noindent with $Y^l$ such that \mbox{$Y^l\!\succeq\!Y$} and \mbox{$\rho(\mathcal{V}^l)\!<\!1$}, for all $l$.
\noindent We will use inductive arguments starting from \mbox{$l\!=\!0$}. Since system \eqref{mathcalG1} is mean square detectable, there exists a mode-dependent filtering gain \mbox{$M^0\!=\![M_n^0]_{n=1}^{\mathbf{I}}$} such that \mbox{$\rho (\mathcal{V}^0)\!<\!1$} and from \cite[Proposition 3.20]{COSTA2005}, there exists a unique \mbox{$Y^0\!\in\!\HH^{\mathbf{I}n_x,+}$}, solution of \eqref{OnMl}, for \mbox{$l\!=\!0$}. From Lemma \ref{lemequations}\eqref{item1}, recalling that \mbox{$\rho(\mathcal{V}^0)\!<\!1$} applying again \cite[Proposition 3.20]{COSTA2005}, 
it follows that \mbox{$Y^0\!\succeq\!Y$}. Assume now that there exists a decreasing sequence sequence \mbox{$\{Y^l\}_{l=0}^{k -1}$}, with \mbox{$Y^l\!\in\!\HH^{\mathbf{I}n_x,+}$}, unique solution of \eqref{OnMl} and 
\beqs
Y^0\!\succeq\!Y^1\!\succeq\!\ldots\! \succeq\!Y^{k-1}\!\!\succeq\!\!Y,\quad \forall\,Y\in\MM,
\eeqs
\noindent $\rho(\mathcal{V}^l) < 1$.\\
\noindent Setting 
\begin{align*}
\tilde{R}_n^{k-1}&\triangleq \tilde{\mathcal{R}}_n(Y^{k - 1}),\\
M_n^k &\triangleq\mathcal{M}_n(Y^{k-1}),\\
A_{n1}^k &\triangleq A + M_n^k L,
\end{align*}
\noindent and applying Lemma \ref{lemequations} \eqref{item2}, the following inequality holds:
\begin{align*}
&Y_n^{k-1}- Y_n -\mathcal{V}_n^k\left(Y_n^{k - 1}-Y\right)\succeq \\
& \hat{\gamma}_n\left(M_n^k-M_n^{k - 1}\right)\tilde{R}_n^{k - 1}\left(M_n^k-M_n^{k - 1}\right)^*.
\end{align*}
Since \mbox{$\tilde{R}_n^{k\!-\!1}\succ \tilde{\mathcal{R}}_n(Y)\succ 0,$} for \mbox{$n\in \mathbb{S}_{\eta},$} we can find \mbox{$\delta^{k - 1} > 0$}, such that \mbox{$\tilde{R}_n^{k-1}\succ \delta^{k-1}\mathbb{I}_{n_x},$}
Thus, we get:
\begin{align*}
& Y_n^{k - 1} - Y_n - \mathcal{V}_n^k(Y^{k - 1}- Y)\succeq \\
& \delta^{k\!-\!1}\hat{\gamma}_n(M_n^{k}- M_n^{k - 1})(M_n^{k} - M_n^{k\!-\!1})^*.
\end{align*}
\noindent Applying {\it Lemma} \ref{lemmathvmathvbar}, \mbox{$\rho (\mathcal{V}^k)\!<\!1$}, and
from \cite[Proposition 3.20]{COSTA2005}, there exists a unique solution \mbox{$Y^k\!\in\!\HH^{\mathbf{I}n_x,+}$} of equation \eqref{OnMl} for \mbox{$l\!=\!k$}. Thus, 
from Lemma 1 \eqref{item3}, it follows that
\begin{align*}
&\left(Y_n^{k- 1} -Y_n^{k}\right) -		\\
&\hat{\gamma}_n\left(A + M_n^k L\right)\mathcal{D}_n(Y^{k - 1} - Y^{k})\left(A + M_n^k L\right)^*+\\
&- (1 - \hat{\gamma}_n)A\mathcal{D}_n(Y^{k - 1} - Y^{k})A^* =\\
&
\hat{\gamma}_n\left(M_n^k - M_n^{k - 1}\right)^*\tilde{R}_n^{k - 1}\left(M_n^k - M_n^{k- 1}\right)^*\succeq 0,
\end{align*}
\noindent and since \mbox{$\rho(\mathcal{V}^k)\! < \!1$}, we get from \cite[ Proposition 3.20]{COSTA2005}, that \mbox{$Y^{k- 1} - Y^{k} \succeq  0$}, i.e. \mbox{$Y^{k - 1} \succeq Y^k\succeq Y$}. This completes the induction argument.
Since \mbox{$\{Y^l\}_{l=0}^{\infty}$} is a decreasing sequence, such that \mbox{$Y^l \succeq Y$}, for all \mbox{$l\!=\!0,1,\ldots,$} we get that there exists \mbox{$Y^+$}, such that (see \cite{JoachimWeidmann1980}, p.79) \mbox{$Y^l\!\to \!Y^+$}, as \mbox{$l\!\to\!\infty$}. Clearly, \mbox{$Y^+\succeq Y$}, for all \mbox{$Y \in \MM$}, because \mbox{$Y$} is arbitrary. Furthermore,
\mbox{$Y_n^{l}$} satisfies \eqref{OnMl}, and taking the limit for \mbox{$l \to \infty$}, we have 
\mbox{$Y^+\!=\!\mathcal{Y}(Y^+)$}. Moreover, \mbox{$\tilde{R}_n(Y^+) \succeq \tilde{R}_n(Y)\!\succ\!0$,} i.e. \mbox{$Y^+\!\in\!\MM$}.\\
\noindent The proof of the Lemma is complete.
\end{proof}
\begin{proof}[Proof of Theorem \ref{thmmaximal}]
From the Schur complement (see \cite[ Lemma 2.23]{COSTA2005}) we have that \mbox{$Y\! \in\! \HH^{\mathbf{I}n_x,*}$,} satisfies \eqref{optpbriccati} if and only if 
\mbox{$
-Y\!+\!\mathcal{Y}(Y)\! \succeq\! 0,
$}
and \mbox{$\tilde{\mathcal{R}}_n(Y)\! \succ \!0,\,n\! \in\!\mathbb{S}_{\eta}$}
that is \mbox{$Y\!\in\!\MM$}. 
Thus, if \mbox{$Y^+\!\in\!\MM$} is such that \mbox{$Y^+\! \succeq\! Y$}, for all\mbox{$Y\! \in\! \MM$}, then \mbox{$\mathrm{tr}\left(Y_1^+\!+\ldots\!+\!Y_{\mathbf{I}}^+\right)\! \geq \!\mathrm{tr}\left(Y_1\!+\!\ldots\!+\!Y_{\mathbf{I}}\right)$,} and it follows that \mbox{$Y^+$} is a solution of the convex programming Problem \ref{mypb}. On the other hand, suppose that \mbox{$\hY\!\in\! \HH^{\mathbf{I}n_x,*}$} is a solution of the {\it Problem} \ref{mypb}, then \mbox{$\hY \!\in \!\MM$}. 
From the optimality of \mbox{$\hY$}, it follows that 
\mbox{
$
\mathrm{tr}(Y_1^+\!-\!\hY_1)
+\mathrm{tr}(Y_{\mathbf{I}}^+\!-\!\hY_{\mathbf{I}})\!\preceq\!0,
$} 
for all \mbox{$Y^+\!\in\!\MM$}.
 Since the system \eqref{mathcalG1} is mean square detectable, from Lemma \ref{lemmaximal}, there exists \mbox{$Y^+\!\!\succeq\!\!\hY$} satisfying \eqref{FiltCARE0}. Therefore,
\mbox{
$
Y_1^+\!-\!\hY_1\!\succeq\! 0
,\ldots,\;Y_{\mathbf{I}}^+\!-\!\hY_{\mathbf{I}}\!\succeq\!0.
$}
The two inequalities above hold if and only if \mbox{$\hY\!=\!Y^+$}. \\\noindent The proof of the theorem is complete.
\end{proof}
\noindent
\begin{proof}[Proof of Theorem \ref{lemuniquestab}]
Assume that \mbox{$\hY=\bmat\hY_n\emat_{n=1}^M$} is a 
stabilizing solution for the Filtering CARE \eqref{FiltCARE0}, i.e. \mbox{$\hY = \mathcal{Y}(\hY)$}, so that system \eqref{mathcalG1} is mean square detectable according to Definition \ref{msddef}. From  Lemma \ref{lemmaximal}, there exists a  maximal solution \mbox{$Y^+ \in \MM$}, satisfying \mbox{$Y^+ = \mathcal{Y}(Y^+)$}. By equality \eqref{item2} of Lemma \ref{lemequations}, the following holds:
\begin{align*}
&\hY_n - Y^+_n \\
&- \hat{\gamma}_n \left(A + \mathcal{M}_n(\hY)L\right)\mathcal{D}_n(\hY\!-\!Y^+)\left(A\!+\!\mathcal{M}_n(\hY) L\right)^*\!-\\
&(1\!-\!\hat{\gamma}_n)A\mathcal{D}_n(\hY\!-\!Y^+)A^*=
\\
&=\hat{\gamma}_n\!\left(\mathcal{M}_n(\hY)\!-\!\mathcal{M}_n(Y^+)\right)\tilde{\mathcal{R}}_n(Y^+)\left(\!\mathcal{M}_n(\hY)\!-\!\mathcal{M}_n(Y^+)\!\right)^*,
\end{align*}
\noindent for all \mbox{$n \in \mathbb{S}_{\eta}$}.
\\\noindent Since \mbox{$\tilde{\mathcal{R}}(Y^+) \succ 0$}, we have
\begin{align*}
\hat{\gamma}_n\!\left(\!\mathcal{M}_n(\hY)\!-\!\mathcal{M}_n(Y^+)\!\right)\!\tilde{\mathcal{R}}_n(Y^+)\!\left(\!\mathcal{M}_n(\hY)\!-\!\mathcal{M}_n(Y^+)\!\right)^*\!\!\succeq\!\!0.
\end{align*}
\noindent Recalling that \mbox{$\hY$} is a stabilizing solution, we have from \cite[Proposition 3.20]{COSTA2005} that \mbox{$\hY- Y^+ \succeq 0$}. But this also implies \mbox{$\tilde{\mathcal{R}}(\hY)\succeq\tilde{\mathcal{R}}(Y^+)\succ 0$}, therefore \mbox{$\hY\in \MM$}. From Lemma \ref{lemmaximal}, \mbox{$\hY - Y^+ \preceq 0$}. The two inequalities above hold if and only if \mbox{$\hY = Y^+$}.\\
\noindent The proof of the theorem is complete.
\end{proof}
\begin{proof}[{\it Proof of Theorem \ref{thmsep}}]
Assume that Markov jump system \eqref{mathcalG1} is mean square stabilizable with one time step delay in the observation of the actuation channel mode according to Definition \ref{Def:MSstab}, and mean square detectable according to Definition \ref{msddef}. Then, by Definition \ref{Def:MSstab} there exists a mode-dependent gain $[F_l]_{l=1}^{\mathbf{N}}$, that stabilizes the dynamics of $x_k$ in the mean square sense, accounting for the one time-step delay observation of the actuation channel mode. By Definition \ref{msddef}, there exists a mode-dependent filtering gain $[M_n]_{n=1}^{\mathbf{I}}$, that stabilizes the error dynamics in the mean square sense, accounting for the current mode observation of the sensing channel. Therefore, by the upper triangular structure of the matrix {\small $\mathbf{\Gamma}$} in \eqref{mathcalGcl}, the closed-loop system dynamics \eqref{mathcalGcl} can be made mean square stable.\\
\addtolength{\textheight}{-20cm}
Assume that the dynamics \eqref{mathcalGcl} can be made mean square stable. Then,  by the upper triangular structure of the matrix $\mathbf{\Gamma}$ in \eqref{mathcalGcl}, there exists a mode-dependent filtering gain $[M_n]_{n=1}^{\mathbf{I}}$, that stabilizes the error dynamics in the mean square sense, accounting for the current mode observation of the sensing channel. Thus, Markov jump system \eqref{mathcalG1} is mean square detectable according to Definition \ref{msddef}. Moreover, there exists a mode-dependent gain $[F_l]_{l=1}^{\mathbf{N}}$, that stabilizes the dynamics of $x_k$ in the mean square sense, accounting for the one time-step delay mode observation of the actuation channel. Therefore, Markov jump system \eqref{mathcalG1} is mean square stabilizable according to Definition \ref{Def:MSstab}.\\\noindent 
The proof of the theorem is complete.
\end{proof}
\end{document}